\documentclass[11pt,a4paper]{article}

\usepackage[utf8]{inputenc}
\usepackage[T1]{fontenc}
\usepackage{amsmath, amssymb, amsthm, mathrsfs}
\usepackage{geometry}
\newtheorem{theorem}{Theorem}[section]
\newtheorem{lemma}[theorem]{Lemma}
\newtheorem{assumption}{Assumption}

\theoremstyle{remark}
\newtheorem{remark}{Remark}

\newcommand{\R}{\mathbf{R}}
\newcommand{\C}{\mathbf{C}}
\newcommand{\Sp}{\mathrm{S}}
\newcommand{\F}{\mathscr{F}}
\newcommand{\rd}{\mathrm{d}}
\newcommand{\norm}[1]{\left\| #1 \right\|}

\begin{document}

\title{Global $L^p$ Second Commutation Lemma}
\author{
Marin Mi\v{s}ur\thanks{Marin Mi\v{s}ur,
University of Zagreb, Faculty of Science, Bijeni\v{c}ka cesta 30,
10000 Zagreb, Croatia.
    \texttt{mmisur@math.hr}}
}
\date{}
\maketitle

\begin{abstract}
We prove the second commutation lemma for the Lebesgue spaces $L^p(\R^d)$, $1 < p < \infty$, on the whole unbounded
domain, extending the $L^2$ theory of Tartar (1990) to the Banach-space framework of H-distributions. Unlike the $L^2$
setting, where the Plancherel isometry and Hilbert-space compactness are available, the global $L^p$ setting has neither.
We control the non-local tail through Calder\'on--Zygmund kernel estimates, an explicit Taylor-remainder identity, and
spatial truncation, obtaining order-$(-\epsilon)$ smoothing of the remainder into the Besov scale. Pairing weakly convergent
sequences with their canonical Nemyckij duals, we then use the lemma to derive $L^p$ transport equations. For first-order
scalar equations we establish the phase-space bicharacteristic (Vlasov) flow of the associated H-distribution when
$p \ge 2$; for the quasilinear $p$-wave system we lift the local energy identity to the microlocal level, obtaining a
Poynting-flux transport of microlocal energy in the linear core $p = 2$ and isolating the structural obstruction to its
closure when $p \neq 2$. These results provide functional-analytic tools for tracking the propagation of singularities in
degenerate nonlinear and non-local partial differential equations.
\end{abstract}



\section{Introduction and Functional Setting}\label{sec:intro}

The second commutation lemma is the analytic tool underlying the propagation principles of H-distributions (and their $L^2$ predecessors, H-measures).
The $L^2$ H-measures were originally developed by Tartar \cite{Tartar1990} and further expanded in his later syntheses on propagation effects \cite{Tartar2017},
while H-distributions were introduced by Antoni\'c and Mitrovi\'c \cite{AntonicMitrovic2011} as their extension to the $L^p$--$L^q$ setting, and have since
found applications to compensated compactness \cite{MisurMitrovic2015} and velocity averaging \cite{LazarMitrovic2012, ErcegMisurMitrovic2023}. 

To define the microlocal energy of sequences in reflexive Banach spaces $L^p(\R^d)$ ($1 < p < \infty$), we must go beyond the $L^2$ restriction of classical
H-measures. The existence of H-distributions in this $L^p$--$L^q$ setting is due to Antoni\'c and Mitrovi\'c \cite{AntonicMitrovic2011}, subsequently refined
by Antoni\'c, Erceg, and Mi\v{s}ur \cite{AntonicErcegMisur2021} to locally weakly convergent sequences and to distributions of anisotropic order; these objects
underlie the $L^p$ commutation theory (see also \cite{Misur2017} for a systematic treatment, \cite{MisurPalle2025}
for further properties of the underlying anisotropic distributions, and \cite{Misur2019} for an application of anisotropic
distributions to a refinement of Peetre's theorem). Whereas that work established the existence and localisation of these objects,
the present paper supplies the second commutation lemma (the identification of the commutator's principal symbol as a Poisson bracket, together
with the global compactness of its remainder) and uses it to derive microlocal transport equations. The commutation lemma itself holds for the full range
$1 < p < \infty$; the transport applications additionally require the pointwise chain rule for the Nemyckij dual, and are therefore proved for $p \ge 2$ in the
first-order scalar case (Section \ref{sec:applications}) and, in the wave case, for the linear core $p = 2$ (Section \ref{sec:wave_eq}). Throughout, the spatial
dimension is $d \ge 2$, so that $\Sp^{d-1}$ is a genuine sphere.

We delineate at the outset what is new here and what is recovered, since the two are deliberately kept distinct. The novel contributions are the global
$L^p$ second commutation lemma (Theorem \ref{thm:second_comm}), namely the identification of the principal symbol as a Poisson bracket together with the compactness
of the remainder on the whole of $\R^d$, obtained without the Plancherel isometry through the truncation argument of Section \ref{sec:theorem_proof}, and the
quantitative order-$(-\epsilon)$ smoothing of that remainder into the Besov scale (Lemma \ref{lem:remainder}, Appendix \ref{app:remainder}). Its principal
application, where the lemma is needed at $p \neq 2$, is the $L^p$ bicharacteristic transport (Vlasov) equation for first-order scalar PDEs
(Theorem \ref{thm:bicharacteristic_first_order}). By contrast, the wave results of Section \ref{sec:wave_eq} are confined to the linear core $p = 2$, where the
H-distribution framework collapses to a classical Tartar H-measure and the commutation lemma is not needed for the constant-coefficient statement
(Theorem \ref{thm:bicharacteristic_pwave}); we include them to fix notation, to exhibit the lemma performing the variable-coefficient refraction computation
(Theorem \ref{thm:wave-variable}), and to isolate, rather than resolve, the structural obstruction to closure when $p \neq 2$
(Remark \ref{rem:nonlinear-obstruction}). We make no claim of novelty for these $L^2$ statements.

\begin{theorem}[Existence of H-distributions \cite{AntonicMitrovic2011, AntonicErcegMisur2021}]\label{thm:h_dist_existence}
Let $1 < p < \infty$ and $p' = p/(p-1)$. Assume that $u_n \rightharpoonup 0$ weakly in $L_{loc}^p(\R^d)$ and $v_n \rightharpoonup 0$ weakly in $L_{loc}^{p'}(\R^d)$.
Then there exist a subsequence and a distribution $\mu \in \mathcal{D}'(\R^d \times \Sp^{d-1})$, uniquely determined by that subsequence, called the H-distribution
associated with the pair $(u_n, v_n)$, such that for all test functions $\varphi_1, \varphi_2 \in C_c^\infty(\R^d)$ and any smooth Fourier multiplier symbol
$a \in C^\infty(\Sp^{d-1})$, the following limit holds:
\begin{equation*}
    \lim_{n \to \infty} \int_{\R^d} \varphi_1(x) u_n(x) \mathcal{A}_a(\varphi_2 v_n)(x) \, \rd x = \langle \mu, \varphi_1 \varphi_2 a \rangle.
\end{equation*}
\end{theorem}

\begin{remark}[Matrix-Valued H-Distributions]
While Theorem \ref{thm:h_dist_existence} is stated for scalar sequences as established by Antoni\'c, Erceg, and Mi\v{s}ur \cite{AntonicErcegMisur2021},
it extends to vector-valued sequences via component-wise application. For sequences
$\mathbf{u}_n \in L_{loc}^p(\R^d; \R^k)$ and $\mathbf{v}_n \in L_{loc}^{p'}(\R^d; \R^m)$, the associated H-distribution $\boldsymbol{\mu}$ is a matrix-valued
distribution in $\mathcal{D}'(\R^d \times \Sp^{d-1}; \R^{k \times m})$, where each entry $\mu_{ij}$ is the scalar H-distribution generated by the pair
$(u_n^i, v_n^j)$; a single subsequence, extracted by a finite diagonal argument, serves all $k \times m$ pairs simultaneously. We use this tensorial
structure in Section \ref{sec:wave_eq} to analyse coupled quasilinear systems.
\end{remark}

Because H-distributions lack the positivity of $L^2$ H-measures (which are Radon measures), they require careful functional analytic treatment.
Establishing propagation rules for these distributions relies on the second commutation lemma, whose analytic core, the compactness of commutators of
multiplication operators with Fourier multipliers, has been studied in \cite{AntonicMisurMitrovic2018, MisurMitrovic2019}. Establishing this lemma in
domains lacking the Plancherel isometry and the compactness tools of Hilbert space requires substantial technical work, comparable in effort
to other departures from the isotropic $L^2$ theory, such as the parabolic H-measures of Antoni\'c and Lazar \cite{AntonicLazar2013}, which
(while remaining within the $L^2$/Plancherel framework) replace the Euclidean sphere by a parabolically scaled manifold. 

To guarantee boundedness and compactness on all of $L^p(\R^d)$ ($1 < p < \infty$), we replace the low-regularity $L^2$ multiplier assumptions with mild
fractional regularity conditions governed by Calder\'on-Zygmund theory.

We consider the Fourier multiplier operator $\mathcal{A}_a$ defined by $\F(\mathcal{A}_a u)(\xi) = a(\xi)\F u(\xi)$, imposing the following conditions:

\begin{assumption}[Symbol Regularity]\label{ass:symbol}
Let $a: \R^d \setminus \{0\} \to \C$ be a homogeneous function of degree zero. We assume the restriction of $a$ to the unit sphere satisfies
$a \in W^{s,2}(\Sp^{d-1})$ for real $s > d+2$.
\end{assumption}

\begin{remark}[On the Sobolev Regularity Threshold]\label{rem:sobolev}
The threshold $s > d+2$ is dictated not by the $L^p$-boundedness of $\mathcal{A}_a$, but by the pointwise Calder\'on--Zygmund kernel bounds on which our
elementary proof relies; two distinct regularity budgets must be kept apart.

\emph{Boundedness budget.} Mere $L^p$-boundedness of $\mathcal{A}_a$ (and of the auxiliary singular integrals $\operatorname{Op}(\xi_j\partial_{\xi_k}a)$
defining the principal symbol $S_j$) follows from the H\"ormander--Mikhlin theorem, which for a symbol homogeneous of degree zero requires only
$s_0 > \frac{d}{2}$ derivatives on the sphere: the radial variable being inert, the full $\R^d$ Mikhlin condition reduces to
$a|_{\Sp^{d-1}} \in W^{s_0,2}(\Sp^{d-1})$ with the same exponent $s_0 > \frac{d}{2}$ (the smaller value $\frac{d-1}{2}$ is only the embedding threshold
$W^{s_0,2}(\Sp^{d-1}) \hookrightarrow L^\infty$, sufficient for $L^2$ boundedness via Plancherel but not for $L^p$). This budget alone would suggest a
threshold near $s > \frac{d+4}{2}$, obtained by adding to $\frac{d}{2}$ the two orders consumed by the second-order Taylor remainder of the commutator.

\emph{Kernel budget (binding).} However, Part I and the kernel identity of Lemma \ref{lem:remainder} (Appendix \ref{app:remainder}) use the pointwise
bounds $|\partial_j k(z)| \le C|z|^{-d-1}$ and $|\nabla\partial_j k(z)| \le C|z|^{-d-2}$ on the kernel of $\partial_j\mathcal{A}_a$. Passing from Sobolev
regularity of the symbol to pointwise kernel-derivative bounds incurs a dimension-dependent loss that the boundedness budget ignores: by the Bochner--Stein--Weiss
identity on spherical harmonics (Appendix \ref{app:kernel}), the symbol-to-kernel transfer for a kernel of degree $-d-1$ costs $\tfrac{d}{2} + 1$ derivatives in
the $W^{\sigma,2}(\Sp^{d-1})$ scale, so the size bound requires $s > d + \tfrac{1}{2}$ and the gradient bound $s > d + \tfrac{3}{2}$. The clean sufficient
condition $s > d+2$ guarantees both in every dimension $d \ge 2$, with roughly half a derivative to spare.
\end{remark}

\begin{remark}[A sharper threshold via Littlewood--Paley]\label{rem:sharper}
The requirement $s > d+2$ is an artefact of our deliberately elementary, kernel-based proof of Lemma \ref{lem:remainder}, and is not sharp.
A frequency-space argument dispenses with pointwise kernel bounds altogether: decomposing the commutator $[\mathcal{A}_a, M_b]$ through the Bony
paraproduct calculus, bounding each dyadic block of $\mathcal{A}_a$ by the H\"ormander--Mikhlin theorem (which costs only $\frac{d}{2}$ derivatives),
and extracting the geometric gain $2^{-\epsilon l}$ across dyadic scales $l$ from the H\"older regularity $\nabla b \in C^\epsilon$, would establish
the same order-$(-\epsilon)$ smoothing under the much weaker hypothesis $s > \frac{d+4}{2}$, the H\"ormander base plus the two orders consumed by the
second-order Taylor remainder. We do not pursue this route: it would require importing the full Littlewood--Paley and paraproduct apparatus, sacrificing
the self-contained character of Appendix \ref{app:remainder}, whereas the symbols arising in all of our applications are $C^\infty(\Sp^{d-1})$, so the
precise value of $s$ is immaterial downstream. We record the sharper threshold only to locate the true regularity demand of the second commutation lemma.
\end{remark}

Throughout, $C_0(\R^d)$ denotes the continuous functions vanishing at infinity in the supremum norm, and $C^\epsilon(\R^d)$ the bounded, uniformly
$\epsilon$-H\"older functions, normed by $\norm{\cdot}_{C^\epsilon(\R^d)} = \norm{\cdot}_{L^\infty} + [\,\cdot\,]_{C^\epsilon}$. Our spatial decay
hypotheses require only that the relevant gradient lie in $C_0(\R^d)$, vanishing at infinity in supremum norm, while remaining uniformly H\"older;
we do not require the H\"older seminorm itself to vanish at infinity (the stronger little-H\"older condition $C_0^\epsilon(\R^d)$, the closure
of $C_c^\infty(\R^d)$ in $\norm{\cdot}_{C^\epsilon(\R^d)}$).

\begin{assumption}[Global Spatial Regularity]\label{ass:spatial_reg}
Let $b \in C^{1,\epsilon}(\R^d)$ for some $\epsilon > 0$. We assume its weak gradient vanishes at infinity in supremum norm, namely
$\nabla b \in C_0(\R^d; \R^d)$. Here $C^{1,\epsilon}$ constrains only the gradient (bounded and uniformly $\epsilon$-H\"older);
no bound on $b$ itself is imposed (its sublinear growth is used in Step~2 of Theorem \ref{thm:second_comm}), and global
boundedness of $b$ is added separately, where needed, as Assumption \ref{ass:spatial_bound}. Only sup-norm decay of $\nabla b$ is required:
the H\"older seminorm $[\nabla b]_\epsilon$ need only be globally bounded, not vanishing at infinity, a strict weakening of the little-H\"older
hypothesis $\nabla b \in C_0^\epsilon$ that the compactness argument would classically invoke (cf.\ Steps~2--3 of the proof of Theorem \ref{thm:second_comm},
where an operator bound in the supremum norm of $\nabla b$ makes sup-norm decay suffice).
\end{assumption}

For the finite-difference (Besov) arguments of Lemmas \ref{lem:remainder} and \ref{lem:building} we take, without loss of generality, $0 < \epsilon < 1$:
a bounded, uniformly $\epsilon$-H\"older function with $\epsilon > 1$ is constant, while any $\epsilon \ge 1$ may be lowered to an arbitrary
$\epsilon'' \in (0,1)$, since a bounded uniformly $\epsilon$-H\"older function is uniformly $\epsilon''$-H\"older for every $\epsilon'' \le \epsilon$
(interpolating the difference quotient against the $L^\infty$ bound at unit length scale), while the sup-norm decay $\nabla b \in C_0$ is preserved.

\section{Theorem Statement and Proof}\label{sec:theorem_proof}

\begin{theorem}[Global $L^p$ Second Commutation Lemma]\label{thm:second_comm}
Under Assumptions \ref{ass:symbol} and \ref{ass:spatial_reg}, the following hold:
\begin{enumerate}
    \item \textbf{Sobolev Mapping:} The commutator $[\mathcal{A}_a, M_b]$ maps $L^p(\R^d)$ boundedly into $\dot{W}^{1,p}(\R^d)$.
    \item \textbf{Principal Symbol and Compactness:} The weak spatial derivative $\partial_j [\mathcal{A}_a, M_b]$ can be decomposed as
    \begin{equation*}
        \partial_j [\mathcal{A}_a, M_b] = S_j + R_j
    \end{equation*}
    where $S_j = \operatorname{Op}\left( \xi_j \sum_{k=1}^d \frac{\partial a}{\partial \xi_k}(\xi) \frac{\partial b}{\partial x_k}(x) \right)$ is the bounded
    pseudo-differential operator defined by the principal symbol, and the remainder $R_j$ is a compact operator on $L^p(\R^d)$.
\end{enumerate}
\end{theorem}

The compactness assertion of Part II rests on the following quantitative smoothing estimate for the remainder, which isolates our single use of the non-smooth
symbolic calculus. We state it separately so that the exact regularity budget of Assumption \ref{ass:symbol} is transparent.

\begin{lemma}[Non-smooth symbolic remainder estimate]\label{lem:remainder}
Let $a$ satisfy Assumption \ref{ass:symbol}, let $0 < \epsilon < 1$, and let $b \in C^{1,\epsilon}(\R^d)$ with $\nabla b \in C^\epsilon(\R^d; \R^d)$.
Then the remainder
\begin{equation*}
    R_j(b) := \partial_j[\mathcal{A}_a, M_b] - S_j, \qquad S_j = \operatorname{Op}\Big( \xi_j \sum_{k=1}^d \tfrac{\partial a}{\partial \xi_k} \tfrac{\partial b}{\partial x_k} \Big),
\end{equation*}
admits the explicit kernel representation
\begin{equation*}
    R_j(b)u(x) = \int_{\R^d} (\partial_j k)(x-y)\,\big[\, b(y) - b(x) - \nabla b(x)\cdot(y-x) \,\big]\, u(y)\, \rd y,
\end{equation*}
where $k = \mathcal{F}^{-1}a$ is the convolution kernel of $\mathcal{A}_a$. It is a smoothing operator of order $-\epsilon$: for every $1 < p < \infty$
it maps $L^p(\R^d)$ boundedly into the Besov space $B^\epsilon_{p,\infty}(\R^d)$, and hence into the fractional Sobolev space $W^{\epsilon',p}(\R^d)$
for every $\epsilon' \in (0,\epsilon)$, with
\begin{equation*}
    \norm{R_j(b)}_{\mathcal{L}(L^p(\R^d),\, B^\epsilon_{p,\infty}(\R^d))} \le C\, \norm{\nabla b}_{C^\epsilon(\R^d)},
\end{equation*}
where $C = C(d,p,\epsilon)$ depends on the symbol only through $\norm{a}_{W^{s,2}(\Sp^{d-1})}$.
\end{lemma}

\begin{proof}
The kernel identity and the boundedness into $B^\epsilon_{p,\infty}(\R^d)$ are established in Appendix \ref{app:remainder} by elementary
means (Calder\'on--Zygmund kernel bounds, Taylor's theorem, and the finite-difference characterisation of Besov spaces), the sole
non-elementary input being the $L^p$-boundedness of $\partial_j[\mathcal{A}_a, M_b]$ from Part I of Theorem \ref{thm:second_comm}
(Calder\'on's first commutator theorem), invoked once to control the $L^p$ component of the Besov norm. The continuous embedding
$B^\epsilon_{p,\infty}(\R^d) \hookrightarrow B^{\epsilon'}_{p,p}(\R^d) = W^{\epsilon',p}(\R^d)$ for $\epsilon' \in (0,\epsilon)$
is standard \cite[\S 2.3.2]{Triebel1983}.
\end{proof}

\begin{proof}[Proof of Part I: Boundedness into $\dot{W}^{1,p}(\R^d)$]
Applying the spatial derivative $\partial_j$ in the sense of Schwartz distributions and using the Leibniz rule yields:
\begin{equation*}
    \partial_j [\mathcal{A}_a, M_b] u = [\mathcal{A}_a \partial_j, M_b] u - (\partial_j b)\mathcal{A}_a u.
\end{equation*}

We evaluate both terms in the $L^p$ norm by verifying the requirements of the underlying harmonic analysis theorems:

\begin{enumerate}
    \item \textbf{The Commutator Term:} The operator $T_j = \mathcal{A}_a \partial_j$ is a singular integral operator of order $1$ with Fourier symbol
    $i\xi_j a(\xi)$, homogeneous of degree $1$. Normalising $a$ to have zero spherical mean (as we may, since $[\mathcal{A}_a, M_b]$ is unchanged by
    $a \mapsto a - c_0$; the mean $c_0$ contributes only the bounded multiplication $[c_0\partial_j, M_b] = c_0 M_{\partial_j b}$; cf.\ Appendix \ref{app:identity}),
    Assumption \ref{ass:symbol} ($s > d+2$; see Remark \ref{rem:sobolev}) makes $i\xi_j a$ restrict to a $C^{\lfloor d/2\rfloor+2}(\Sp^{d-1})$ function;
    hence its convolution kernel $K_j = \partial_j k$ is homogeneous of degree $-d-1$ and smooth off the origin (Grafakos \cite[Proposition 2.4.8]{Grafakos2014})
    and satisfies both the size bound $|K_j(z)| \le C|z|^{-d-1}$ and the regularity bound $|\nabla K_j(z)| \le C|z|^{-d-2}$ of \eqref{eq:app-kernel}.
    Consequently the commutator kernel $L(x,y) = K_j(x-y)\big(b(x)-b(y)\big)$ is a Calder\'on--Zygmund kernel of order $0$: with $b$ globally Lipschitz
    (established below), $|L(x,y)| \le C\norm{\nabla b}_\infty|x-y|^{-d}$ and $|\nabla_{x,y}L(x,y)| \le C\norm{\nabla b}_\infty|x-y|^{-d-1}$,
    while the cancellation and weak-boundedness conditions hold automatically: $K_j$ has divergence structure, so $T_j 1 = 0$ and
    $[T_j, M_b]1 = \mathcal{A}_a(\partial_j b) \in \mathrm{BMO}$ with norm $\le C\norm{\nabla b}_\infty$ (symmetrically for the transpose).
    Calder\'on's First Commutator Theorem \cite{Calderon1965} for the odd part of $K_j$, and Coifman--Meyer \cite{CoifmanMeyer1978} (equivalently
    the $T(1)$ theorem; see also Stein \cite[Chapter VI]{Stein1993}) for the general complex kernel, then apply once the spatial multiplier is globally Lipschitz. 
    
    By Assumption \ref{ass:spatial_reg}, $\nabla b \in C_0(\R^d) \subset L^\infty(\R^d; \R^d)$, so that $b$ is globally Lipschitz (i.e.,
    $b \in \dot{W}^{1,\infty}(\R^d)$). All hypotheses (kernel size, kernel regularity, cancellation, and multiplier Lipschitz continuity) being met,
    Calder\'on's theorem guarantees that $[T_j, M_b]$ is a bounded operator on $L^p(\R^d)$.

    \item \textbf{The Multiplication Term:} To bound the action of the Fourier multiplier $\mathcal{A}_a$, we invoke the H\"ormander multiplier
    theorem \cite{Hormander1960} (see also Grafakos \cite[Theorem 6.2.7]{Grafakos2014}). The classical theorem requires the symbol to satisfy
    sufficient fractional differentiability bounds. Because $a(\xi)$ is homogeneous of degree zero, its radial derivatives vanish, meaning the
    H\"ormander condition reduces entirely to its regularity on the unit sphere, specifically requiring $a \in W^{s_0,2}(\Sp^{d-1})$ for $s_0 > \frac{d}{2}$. 
    
    Our Assumption \ref{ass:symbol} exceeds this threshold, confirming that $\mathcal{A}_a$ is a bounded operator on $L^p(\R^d)$. 
    
    Finally, as established above, $\partial_j b \in L^\infty(\R^d)$. Because $L^p(\R^d)$ is a module over $L^\infty(\R^d)$, pointwise multiplication
    by $\partial_j b$ preserves $L^p$ integrability:
    \begin{equation*}
        \norm{(\partial_j b)\mathcal{A}_a u}_{L^p} \le \norm{\partial_j b}_{L^\infty} \norm{\mathcal{A}_a u}_{L^p} \le C \norm{u}_{L^p}.
    \end{equation*}
\end{enumerate}

Consequently, both components are bounded in $L^p$, proving that the commutator $[\mathcal{A}_a, M_b]: L^p(\R^d) \to \dot{W}^{1,p}(\R^d)$ acts as a bounded
linear operator. (When $p \ge d$ the kernel integral defining $[\mathcal{A}_a, M_b]$ need not converge absolutely for every $u \in L^p$; the operator is
then defined on $C_c^\infty(\R^d)$ and extended by the bound just proved, its image lying in $\dot{W}^{1,p}(\R^d)$ read modulo constants. This is immaterial
below: Part II uses only the operator $\partial_j[\mathcal{A}_a, M_b]$, and the applications pair against compactly supported test functions.)
\end{proof}

\begin{proof}[Proof of Part II: Global Compactness via Truncation]
In standard pseudo-differential calculus, spatial symbols are assumed to be infinitely differentiable. Because our spatial multiplier possesses only limited
fractional regularity ($b \in C^{1,\epsilon}(\R^d)$), we work in the non-smooth symbolic setting of Coifman and Meyer \cite{CoifmanMeyer1978} (see also
Taylor \cite[Chapter 1; the relevant commutator estimates are in \S 3.6 and \S 4.1]{Taylor1991} for a modern exposition); the single quantitative input we
require from it is furnished (by an elementary, self-contained Calder\'on--Zygmund kernel argument rather than the full paraproduct calculus) in
Lemma \ref{lem:remainder}. That lemma establishes that the remainder operator $R_j(b) = \partial_j [\mathcal{A}_a, M_b] - S_j(b)$ is a smoothing operator
of order $-\epsilon$; fixing once and for all an exponent $\epsilon' \in (0,\epsilon)$, it maps $L^p(\R^d)$ boundedly into the fractional Sobolev space
$W^{\epsilon', p}(\R^d)$, with operator norm controlled by $\norm{\nabla b}_{C^\epsilon(\R^d)}$. This fractional smoothing is the sole ingredient we extract
from the non-smooth calculus; the remainder of the argument is elementary.

\vspace{0.5cm}
\textbf{Step 1: Compactness for Compactly Supported Multipliers.} 

First, assume $b_c \in C_c^{1,\epsilon}(\R^d)$ is a multiplier with compact support. We must show $R_j(b_c)$ is a compact operator on $L^p(\R^d)$.
To invoke the fractional Rellich-Kondrachov theorem \cite[Theorem 7.1]{DiNezza2012}, the target Sobolev space must be restricted to a bounded spatial
domain $\Omega \subset \R^d$. 

The global embedding $W^{\epsilon',p}(\R^d) \hookrightarrow L^p(\R^d)$ is continuous but fails to be compact on the unbounded domain, since uniformly
bounded sequences may escape to spatial infinity without converging. Although the multiplier $b_c$ has compact support, singular integral and
pseudo-differential operators are non-local; thus, the mapped functions $R_j(b_c) u$ do not inherit this compact support, preventing a direct, global
application of the embedding theorem. 

To handle this, we isolate the locally compact behaviour from the non-local operator tail. Let $\tilde{\chi} \in C_c^\infty(\R^d)$ be a cutoff function
such that $\tilde{\chi} \equiv 1$ on the support of $b_c$ and vanishes outside a larger ball. We decompose the operator:
\begin{equation*}
    R_j(b_c) = \tilde{\chi} R_j(b_c) + (1-\tilde{\chi}) R_j(b_c).
\end{equation*}

\textit{The Local Part:} The operator $\tilde{\chi} R_j(b_c)$ maps $L^p(\R^d) \to W^{\epsilon',p}(\R^d)$. The left multiplication by $\tilde{\chi}$
ensures the image is supported within the compact domain $\Omega = \text{supp}(\tilde{\chi})$. By the compact embedding theorem for fractional Sobolev
spaces \cite[Theorem 7.1]{DiNezza2012}, the embedding $W^{\epsilon',p}(\Omega) \hookrightarrow L^p(\Omega)$ is compact. Thus, $\tilde{\chi} R_j(b_c)$
is a compact operator.

\textit{The Non-Local Tail:} We evaluate the action of $(1-\tilde{\chi}) R_j(b_c) u$ for $x \notin \text{supp}(\tilde{\chi})$. In this region, $b_c(x) = 0$
and $\nabla b_c(x) = 0$. Consequently, $S_j(b_c) u(x) = 0$ and $M_{b_c} \mathcal{A}_a u(x) = 0$. The remainder reduces to the action of the integral operator:
\begin{equation*}
    (1-\tilde{\chi}(x)) R_j(b_c) u(x) = (1-\tilde{\chi}(x)) \partial_j \mathcal{A}_a (b_c u) (x) = \int_{\R^d} H(x,y) u(y) \rd y
\end{equation*}
where the integral kernel is $H(x,y) = (1-\tilde{\chi}(x)) K_j(x-y) b_c(y)$, and $K_j$ is the convolution kernel of $\partial_j \mathcal{A}_a$. 

Because $\text{supp}(1-\tilde{\chi})$ and $\text{supp}(b_c)$ are separated by a strictly positive distance $M > 0$, the singularity of $K_j$ at the origin is
avoided. Standard Calder\'on-Zygmund kernel estimates (see, e.g., Grafakos \cite[Proposition 2.4.8]{Grafakos2014} or Stein \cite[Chapter VI]{Stein1993})
guarantee that $|K_j(z)| \le C|z|^{-d-1}$. Therefore, the kernel satisfies:
\begin{equation*}
    |H(x,y)| \le C (1+|x|)^{-d-1} \mathbf{1}_{\text{supp}(b_c)}(y).
\end{equation*}

This bounded, smooth kernel has rapid decay in $x$ and compact support in $y$. Integrating $|H(x,y)|^p$ in $x$ yields an asymptotic tail bounded by
$\int (1+|x|)^{-p(d+1)} \rd x$, which converges since $p(d+1) > d$ for any $p > 1$. This integrability ensures the kernel belongs to the mixed Lebesgue
space $L^p(\R^d_x; L^{p'}(\R^d_y))$. 

Integral operators satisfying this Hille-Tamarkin condition are compact on $L^p(\R^d)$ (see, e.g., J\"orgens \cite{Jorgens1982}). Since both the local part
and the non-local tail are compact, $R_j(b_c)$ is a compact operator.

\vspace{0.5cm}
\textbf{Step 2: Truncation Convergence.}

To pass from locally compact operators to the global remainder $R_j(b)$, we construct a sequence of compactly supported multipliers $\{b_n\}$ such that their
corresponding operators $R_j(b_n)$ converge to $R_j(b)$ in the uniform operator topology. By Assumption \ref{ass:spatial_reg}, $b \in C^{1,\epsilon}(\R^d)$
with the sup-norm vanishing condition $\nabla b \in C_0(\R^d)$. 

Let $\chi \in C_c^\infty(\R^d)$ be a standard smooth localizing bump function satisfying $\chi \equiv 1$ on $B(0,1)$ and $\text{supp}(\chi) \subset B(0,2)$.
We define the scaled spatial cutoff $\chi_n(x) = \chi(x/n)$ and construct the truncated multiplier sequence as:
\begin{equation*}
    b_n(x) = (b(x) - b(0))\chi_n(x).
\end{equation*}

\textit{Note:} We deliberately truncate $(b(x) - b(0))$ rather than $b(x)$. Because pseudo-differential commutators annihilate constants (i.e., $[P, M_{c}] = 0$),
$R_j(b)$ depends exclusively on the gradient $\nabla b$, not on the absolute magnitude of $b$. Subtracting $b(0)$ anchors the function at the origin, which is
needed to control its asymptotic growth rate via the fundamental theorem of calculus. 

Applying the product rule yields the gradient difference:
\begin{equation*}
    \nabla b_n(x) - \nabla b(x) = (\chi_n(x) - 1) \nabla b(x) + E_n(x),
\end{equation*}
where the error term generated by the cutoff derivative is $E_n(x) = \frac{1}{n} (\nabla \chi)(x/n) (b(x) - b(0))$. 

We analyze both terms in the $L^\infty(\R^d)$ topology. Unlike the classical little-H\"older route, no control of the H\"older seminorm of
$\nabla b_n - \nabla b$ is needed: the operator estimate invoked in Step~3 depends on the multiplier gradient only through its supremum norm
(Appendix \ref{app:lp}, estimate \eqref{eq:app-lpbound}).

First, the cutoff remainder $(\chi_n - 1)\nabla b$ vanishes uniformly. Since $\chi_n \equiv 1$ on $B(0,n)$ and $|\chi_n - 1| \le 1$ globally,
\begin{equation*}
    \norm{(\chi_n - 1)\nabla b}_{L^\infty} \le \sup_{|x| \ge n} |\nabla b(x)| \longrightarrow 0 \qquad (n \to \infty),
\end{equation*}
by the sup-norm decay $\nabla b \in C_0(\R^d)$ of Assumption \ref{ass:spatial_reg}.

Second, we bound the error term $E_n$. Fix $\delta > 0$. Since $\nabla b \in C_0(\R^d)$, choose $R$ with $|\nabla b(x)| < \delta$ for $|x| > R$;
then for $|x| > R$ the fundamental theorem of calculus $b(x) - b(0) = \int_0^1 \nabla b(tx)\cdot x\,\rd t$, split at $t = R/|x|$ (where $|tx| < R$
contributes $\norm{\nabla b}_{L^\infty}$ over a set of $t$-measure $R/|x|$, and $|tx| > R$ contributes at most $\delta$), yields the explicit
sublinear bound $|b(x) - b(0)| \le R\,\norm{\nabla b}_{L^\infty} + \delta|x|$. The cutoff derivative $(\nabla \chi)(x/n)$ is supported in the
annulus $n \le |x| \le 2n$, so for $n > R$,
\begin{equation*}
    \norm{E_n}_{L^\infty} \le \frac{\norm{\nabla \chi}_{L^\infty}}{n} \sup_{n \le |x| \le 2n} |b(x) - b(0)|
    \le \frac{\norm{\nabla \chi}_{L^\infty}}{n}\big(R\,\norm{\nabla b}_{L^\infty} + 2n\delta\big)
    \le \norm{\nabla \chi}_{L^\infty}\Big(\tfrac{R\,\norm{\nabla b}_{L^\infty}}{n} + 2\delta\Big).
\end{equation*}
Letting $n \to \infty$ gives $\limsup_n \norm{E_n}_{L^\infty} \le 2\delta\,\norm{\nabla \chi}_{L^\infty}$; as $\delta > 0$ was arbitrary,
$\norm{E_n}_{L^\infty} \to 0$.

Summing the two estimates yields $\lim_{n \to \infty} \norm{\nabla b_n - \nabla b}_{L^\infty(\R^d)} = 0$. No H\"older-seminorm or interpolation estimate is
required: the deliberately weakened decay hypothesis of Assumption \ref{ass:spatial_reg} is calibrated precisely so that supremum-norm convergence suffices.

\vspace{0.5cm}
\textbf{Step 3: Uniform Operator Limit.}

To conclude global compactness, we evaluate the topological limit of the truncated operators $R_j(b_n)$. We first observe the algebraic behavior of the
commutator mapping $b \mapsto \partial_j [\mathcal{A}_a, M_b]$. This mapping is linear in the spatial multiplier. Furthermore, because Fourier multipliers
are linear operators that commute with scalar multiplication, the commutator annihilates constants: for any $c \in \C$,
$[\mathcal{A}_a, M_c]u = \mathcal{A}_a(cu) - c\mathcal{A}_a u = 0$. 

Equally important, the principal pseudo-differential operator $S_j(b)$ also annihilates constants, as its spatial amplitude contains the gradient $\nabla b$
as a linear factor. Consequently, the remainder operator $R_j(b) = \partial_j [\mathcal{A}_a, M_b] - S_j(b)$ vanishes on constants, so its action depends only
on the spatial variations of $b$, quantified by its gradient $\nabla b$.

A distinction must be made regarding sublinear growth. Because $b(x)$ grows at infinity, the truncated difference $b(x) - b_n(x)$ is not in $L^\infty(\R^d)$,
so a pseudo-differential estimate requiring full symbol control would fail here. 

However, the remainder obeys the stronger operator bound \eqref{eq:app-lpbound} of Appendix \ref{app:lp},
\begin{equation*}
    \norm{R_j(\beta)}_{\mathcal{L}(L^p)} \le C\, \norm{\nabla \beta}_{L^\infty(\R^d)},
\end{equation*}
valid for every $\beta$ with $\nabla \beta \in L^\infty$: it follows at once from the Calder\'on-commutator bound of Part~I,
$\norm{\partial_j[\mathcal{A}_a, M_\beta]}_{\mathcal{L}(L^p)} \le C\norm{\nabla \beta}_\infty$, together with
$\norm{S_j(\beta)}_{\mathcal{L}(L^p)} \le C\norm{\nabla \beta}_\infty$, and depends on $\beta$ only through $\nabla \beta$, thereby
bypassing the (unbounded) $L^\infty$ amplitude of $b - b_n$. Thus the mapping $\nabla b \mapsto R_j(b)$ is a bounded linear transformation
from $L^\infty(\R^d; \R^d)$ into the Banach algebra $\mathcal{L}(L^p(\R^d))$, and
\begin{equation*}
    \norm{R_j(b) - R_j(b_n)}_{\mathcal{L}(L^p)} = \norm{R_j(b - b_n)}_{\mathcal{L}(L^p)} \le C \norm{\nabla b - \nabla b_n}_{L^\infty(\R^d)}.
\end{equation*}

By the truncation convergence established in Step 2, the right-hand side vanishes as $n \to \infty$. This implies that
$\lim_{n \to \infty} \norm{R_j(b) - R_j(b_n)}_{\mathcal{L}(L^p)} = 0$, so the sequence $\{R_j(b_n)\}$ converges to $R_j(b)$ in the uniform operator topology.

Finally, we invoke a standard theorem of functional analysis: for any Banach space $X$ (here, $X = L^p(\R^d)$), the space of compact operators $\mathcal{K}(X)$
forms a closed ideal within the space of bounded operators $\mathcal{L}(X)$ under the uniform operator norm. By Step 1, every locally truncated operator
$R_j(b_n)$ is a compact operator on $L^p(\R^d)$. 

Because $R_j(b)$ is the uniform operator limit of a sequence of compact operators, the closedness of $\mathcal{K}(L^p)$ guarantees that the global remainder
$R_j(b)$ is compact on $L^p(\R^d)$. This completes the proof.
\end{proof}

\begin{remark}[Uniformity in an Evolution Parameter]\label{rem:parameter}
Theorem \ref{thm:second_comm} is stated for a spatial multiplier $b$ depending on $x \in \R^d$ alone. In the applications of Sections \ref{sec:applications}
and \ref{sec:wave_eq}, the Fourier multiplier $\mathcal{A}_a$ acts solely in the spatial variable $x$, while the multipliers additionally depend on an
evolution parameter $t \in \R$. Since the constants appearing in Parts I and II depend on $b$ only through $\norm{\nabla b}_{C^\epsilon(\R^d)}$ (and,
for the compactness of $R_j$, only through the uniform modulus of decay of $\nabla b$ at spatial infinity encoded in Assumption \ref{ass:spatial_reg}),
the decomposition $\partial_j[\mathcal{A}_a, M_{b(t,\cdot)}] = S_j(t) + R_j(t)$ holds for each fixed $t$, with the compactness of $R_j(t)$ and all operator
bounds uniform in $t$ whenever the hypotheses on $b(t,\cdot)$ hold uniformly in $t$. We invoke the lemma in this parametrised form throughout Sections
\ref{sec:applications} and \ref{sec:wave_eq}.
\end{remark}

\begin{assumption}[Global Spatial Boundedness]\label{ass:spatial_bound}
In addition to Assumption \ref{ass:spatial_reg}, we assume the spatial multiplier is globally bounded, namely $b \in L^\infty(\R^d)$.
\end{assumption}

\begin{theorem}[Inhomogeneous $L^p$ Boundedness]\label{thm:inhomogeneous}
Under Assumptions \ref{ass:symbol}, \ref{ass:spatial_reg}, and \ref{ass:spatial_bound}, the commutator $[\mathcal{A}_a, M_b]$ maps $L^p(\R^d)$ boundedly
into the inhomogeneous Sobolev space $W^{1,p}(\R^d)$.
\end{theorem}

\begin{proof}
By Theorem \ref{thm:second_comm}, we have already established that the spatial gradient is bounded:
$\norm{\partial_j [\mathcal{A}_a, M_b] u}_{L^p} \le C\norm{u}_{L^p}$. To establish mapping into the inhomogeneous space $W^{1,p}(\R^d)$,
it remains to prove the zero-order $L^p$ bound of the commutator itself. Expanding the commutator yields:
\begin{equation*}
    [\mathcal{A}_a, M_b]u = \mathcal{A}_a(bu) - b\mathcal{A}_a u.
\end{equation*}

Taking the $L^p(\R^d)$ norm and applying the triangle inequality, we evaluate both terms. For the first term, because $\mathcal{A}_a$ is bounded on
$L^p(\R^d)$ via the H\"ormander condition (Assumption \ref{ass:symbol}), and $b \in L^\infty(\R^d)$ (Assumption \ref{ass:spatial_bound}), we have:
\begin{equation*}
    \norm{\mathcal{A}_a(bu)}_{L^p} \le C_1 \norm{bu}_{L^p} \le C_1 \norm{b}_{L^\infty} \norm{u}_{L^p}.
\end{equation*}

For the second term, pointwise multiplication by $b \in L^\infty(\R^d)$ boundedly maps $L^p(\R^d)$ to itself:
\begin{equation*}
    \norm{b\mathcal{A}_a u}_{L^p} \le \norm{b}_{L^\infty} \norm{\mathcal{A}_a u}_{L^p} \le C_2 \norm{b}_{L^\infty} \norm{u}_{L^p}.
\end{equation*}

Summing these bounds confirms that $\norm{[\mathcal{A}_a, M_b]u}_{L^p} \le C \norm{u}_{L^p}$. Since both the operator and its weak spatial derivatives
are bounded in $L^p(\R^d)$, the commutator maps continuously into $W^{1,p}(\R^d)$. This completes the proof.
\end{proof}

\section{Generalization to Standard Operators in $L^p$}\label{sec:generalization}

Tartar established that the second commutation lemma extends to \emph{standard operators} via algebraic manipulations. Following Tartar \cite[\S 1]{Tartar2017},
a standard operator is a finite sum $S = \sum_m \mathcal{A}_{a_m} M_{b_m}$ of compositions of a Fourier multiplier $\mathcal{A}_{a_m}$ (symbol $a_m$ homogeneous
of degree zero) with a multiplication $M_{b_m}$; its symbol is $s(x,\xi) = \sum_m a_m(\xi)\,b_m(x)$. When the spatial coefficients decay at infinity
($b_m \in C^\epsilon \cap C_0$), so that the first commutators $[\mathcal{A}_{a_m}, M_{b_m}]$ are compact (Lemma \ref{lem:building}(2)), $S$ is determined
by $s$ modulo compact operators, an \emph{operator of symbol $s$} being any operator differing from $S$ by a compact one, and in particular the reversed
product $\sum_m M_{b_m}\mathcal{A}_{a_m}$ then has the same symbol. For the general coefficients admitted below ($b_m \in L^\infty$ with $\nabla b_m \in C_0$
but $b_m \notin C_0$) this representation-independence fails, the first commutators need not be compact, and we accordingly work throughout with the
fixed ordered representations $\mathcal{A}_a M_b$, never appealing to symbol calculus modulo compacts for the operators themselves. Having proven the
$L^p$ compactness of the \emph{second-commutator remainder}
$R_j = \partial_j[\mathcal{A}_a, M_b] - S_j$ (Theorem \ref{thm:second_comm}(II)), and, for coefficients decaying at infinity, of the first
commutator $[\mathcal{A}_a, M_g]$ itself when $g \in C^\epsilon \cap C_0$ (Lemma \ref{lem:building}(2)), we mirror this generalization in the Banach
space setting using the algebraic ideal of compact operators. The undifferentiated commutator $[\mathcal{A}_a, M_b]$ need
not be compact for a general $b$ with $\nabla b \in C_0$ but $b \notin C_0$: by Theorem \ref{thm:inhomogeneous} it maps only into $W^{1,p}(\R^d)$,
whose embedding into $L^p(\R^d)$ is non-compact, and e.g.\ $b(x) = \sin(\log|x|)$ (smoothed near the origin), bounded and with $\nabla b \in C_0$, lies
in $\mathrm{BMO}\setminus\mathrm{VMO}$, so its Riesz-transform commutators are bounded but not compact. The construction below accordingly never uses
compactness of $[\mathcal{A}_a, M_b]$; it uses only the remainder $R_j$ and the first commutator with $C_0$ coefficients.
Let $S_1$ and $S_2$ be finite sums of standard operators of the form:
\begin{equation*}
    S_1 = \sum_{m} \mathcal{A}_{a_m} M_{b_m}, \quad S_2 = \sum_{k} \mathcal{A}_{c_k} M_{d_k}.
\end{equation*}

The compactness of the remainder rests on two elementary building blocks, which isolate the two ways a compact operator arises here: fractional smoothing
followed by multiplication by a decaying coefficient, and the commutator of a Fourier multiplier with a decaying H\"older multiplier.

\begin{lemma}[Compactness building blocks]\label{lem:building}
Let $e \in C^\infty(\Sp^{d-1})$ be homogeneous of degree zero, and $1 < p < \infty$.
\begin{enumerate}
    \item If $T \in \mathcal{L}\big(L^p(\R^d), W^{\sigma,p}(\R^d)\big)$ for some $\sigma > 0$ and $g \in C_0(\R^d)$, then $M_g T$ is compact on $L^p(\R^d)$.
    \item For $g \in C^\epsilon(\R^d) \cap C_0(\R^d)$ (bounded, uniformly $\epsilon$-H\"older, and vanishing at infinity in supremum norm), the commutator
    $[\mathcal{A}_e, M_g]$ is compact on $L^p(\R^d)$.
\end{enumerate}
\end{lemma}

\begin{proof}
(1) Let $\{u_n\}$ be bounded in $L^p$; then $\{T u_n\}$ is bounded in $W^{\sigma,p}$. Given $\eta > 0$, choose $R$ with $|g| < \eta$ on $\{|x| > R\}$.
On the ball $B_R$ the embedding $W^{\sigma,p}(B_R) \hookrightarrow L^p(B_R)$ is compact (by the classical Rellich--Kondrachov theorem when
$\sigma \ge 1$, and by its fractional counterpart \cite[Theorem 7.1]{DiNezza2012} when $\sigma \in (0,1)$), so a subsequence of $\mathbf{1}_{B_R} T u_n$
converges in $L^p$, and multiplication by $g \in L^\infty$ preserves convergence; outside $B_R$,
$\norm{(1 - \mathbf{1}_{B_R}) g\, T u_n}_{L^p} \le \eta \sup_n\norm{T u_n}_{L^p}$.
A diagonal argument over $\eta \to 0$ shows $\{M_g T u_n\}$ is precompact, so $M_g T$ is compact.

(2) The commutator has kernel $k_e(x-y)\big(g(y) - g(x)\big)$, where $k_e$ is the Calder\'on--Zygmund kernel of $\mathcal{A}_e$, $|k_e(z)| \le C|z|^{-d}$;
since $|g(y)-g(x)| \le [g]_{C^\epsilon}|x-y|^\epsilon$ it is bounded by $C[g]_{C^\epsilon}|x-y|^{-d+\epsilon}$. The finite-difference argument of
Appendix \ref{app:remainder} applies to this order-$(-\epsilon)$ kernel, indeed more simply: the increment of the commutator kernel $k_e(x-y)(g(y)-g(x))$
under $x\mapsto x+h$ splits as $[k_e(x+h-y)-k_e(x-y)]\,(g(y)-g(x+h)) + k_e(x-y)\,(g(x)-g(x+h))$, whose first term obeys the Schur bound
$\le C[g]_{C^\epsilon}|h|^\epsilon$ (mean value theorem for $k_e$ together with $|g(y)-g(x+h)| \le [g]_{C^\epsilon}|x+h-y|^\epsilon$) and whose second term is
$-\Delta_h g(x)$ times the uniformly bounded truncated Calder\'on--Zygmund operator $\mathcal{A}_e$, hence of operator norm $O(|h|^\epsilon)$; this yields
$[\mathcal{A}_e, M_g] \in \mathcal{L}(L^p, B^\epsilon_{p,\infty})$ with norm $\le C\norm{g}_{C^\epsilon}$ (the $L^p$ component of the Besov norm being
controlled by $\norm{[\mathcal{A}_e, M_g]}_{\mathcal{L}(L^p)} \le 2\norm{\mathcal{A}_e}_{\mathcal{L}(L^p)}\norm{g}_{L^\infty}$, the finite-difference
component by $C[g]_{C^\epsilon}$), hence $[\mathcal{A}_e, M_g] \in \mathcal{L}(L^p, W^{\epsilon',p})$ for $\epsilon' \in (0,\epsilon)$.
For $g \in C_c^\epsilon$ of compact support the operator is compact by the truncation of Step 1 in the proof of Theorem \ref{thm:second_comm}: the local part
$M_{\tilde\chi}[\mathcal{A}_e, M_g]$ is compact by part (1) with $\sigma = \epsilon'$, and the non-local tail is a Hille--Tamarkin operator with kernel
$|(1-\tilde\chi(x))k_e(x-y)g(y)| \le C(1+|x|)^{-d}\mathbf{1}_{\operatorname{supp}g}(y) \in L^p_x(L^{p'}_y)$ (as $dp > d$). For general
$g \in C^\epsilon \cap C_0$, the truncations $g_n = g\chi_n \in C_c^\epsilon$ satisfy
$\norm{g - g_n}_{L^\infty} = \norm{g(1-\chi_n)}_{L^\infty} \le \sup_{|x| \ge n}|g(x)| \to 0$ by the sup-norm decay $g \in C_0$;
since $\norm{[\mathcal{A}_e, M_\gamma]}_{\mathcal{L}(L^p)} \le 2\norm{\mathcal{A}_e}_{\mathcal{L}(L^p)}\norm{\gamma}_{L^\infty}$ for every
$\gamma \in L^\infty$, it follows that $[\mathcal{A}_e, M_{g_n}] \to [\mathcal{A}_e, M_g]$ in operator norm; as $\mathcal{K}(L^p)$ is closed,
$[\mathcal{A}_e, M_g]$ is compact. As in Theorem \ref{thm:second_comm}, the uniform H\"older seminorm $[g]_\epsilon$ is consumed only by the
compact-support smoothing above, while the passage to the global limit uses solely the sup-norm decay of $g$.
\end{proof}

\begin{remark}[Overlap with the first commutation lemma]\label{rem:firstcomm}
Part (2) (compactness of $[\mathcal{A}_e, M_g]$ on $L^p(\R^d)$ for $g \in C^\epsilon \cap C_0$) is the $L^p$ first commutation lemma,
whose compactness conclusion is already available in \cite{AntonicMisurMitrovic2018, MisurMitrovic2019}. What Lemma \ref{lem:building}(2) adds is the
quantitative order-$(-\epsilon)$ smoothing $[\mathcal{A}_e, M_g] \in \mathcal{L}(L^p, B^\epsilon_{p,\infty})$, which lets us organise the compactness
of the second-commutator remainder uniformly through Lemma \ref{lem:remainder}.
\end{remark}

\begin{theorem}[$L^p$ Commutation of Standard Operators]\label{thm:standard_op}
Assume all Fourier symbols $a_m, c_k \in C^\infty(\Sp^{d-1})$ are homogeneous of degree zero, and all spatial multipliers $b_m, d_k$ satisfy
Assumptions \ref{ass:spatial_reg} and \ref{ass:spatial_bound}. Then, the weak spatial derivative of the commutator $\partial_j [S_1, S_2]$
boundedly maps $L^p(\R^d)$ into $L^p(\R^d)$ and decomposes as:
\begin{equation*}
    \partial_j [S_1, S_2] = \operatorname{Op}(\xi_j \{s_1, s_2\}) + R_{1,2}
\end{equation*}
where $s_1(x,\xi) = \sum_m a_m(\xi)b_m(x)$ and $s_2(x,\xi) = \sum_k c_k(\xi)d_k(x)$ are the symbols of $S_1$ and $S_2$, $\{\cdot, \cdot\}$ denotes the
Poisson bracket, and $R_{1,2}$ is a compact operator on $L^p(\R^d)$.
\end{theorem}

\begin{remark}
Requiring the Fourier symbols to be smooth is the standard setting for Tartar's standard operators and is no restriction in practice: in every application
the symbols arise from Fourier multipliers with $a_m, c_k \in C^\infty(\Sp^{d-1})$, and it is only the spatial coefficients $b_m, d_k$, the rough PDE coefficients,
that carry the low $C^{1,\epsilon}$ regularity. Smoothness of the symbols guarantees that the composite multipliers $\mathcal{A}_c B_l^a$ appearing below again
satisfy Assumption \ref{ass:symbol}, so that Theorem \ref{thm:second_comm} applies to them without any loss of derivatives.
\end{remark}

\begin{proof}
By bilinearity, it suffices to prove the theorem for single terms: $S_1 = \mathcal{A}_a M_b$ and $S_2 = \mathcal{A}_c M_d$. We expand the commutator
algebraically, using the fact that Fourier multipliers commute with each other ($\mathcal{A}_a \mathcal{A}_c = \mathcal{A}_c \mathcal{A}_a$) and
spatial multiplications commute with each other ($M_b M_d = M_d M_b$). 

By inserting commutators, we rewrite the operator products:
\begin{equation*}
    M_b \mathcal{A}_c = \mathcal{A}_c M_b - [\mathcal{A}_c, M_b] \quad \text{and} \quad M_d \mathcal{A}_a = \mathcal{A}_a M_d - [\mathcal{A}_a, M_d].
\end{equation*}

Substituting these into the expansion of $[S_1, S_2]$ yields:
\begin{align*}
    [S_1, S_2] &= \mathcal{A}_a (M_b \mathcal{A}_c) M_d - \mathcal{A}_c (M_d \mathcal{A}_a) M_b \\
               &= \mathcal{A}_a (\mathcal{A}_c M_b - [\mathcal{A}_c, M_b]) M_d - \mathcal{A}_c (\mathcal{A}_a M_d - [\mathcal{A}_a, M_d]) M_b \\
               &= \mathcal{A}_c [\mathcal{A}_a, M_d] M_b - \mathcal{A}_a [\mathcal{A}_c, M_b] M_d.
\end{align*}

Because Fourier multipliers commute with the weak spatial derivative $\partial_j$, we obtain:
\begin{equation*}
    \partial_j [S_1, S_2] = \mathcal{A}_c (\partial_j [\mathcal{A}_a, M_d]) M_b - \mathcal{A}_a (\partial_j [\mathcal{A}_c, M_b]) M_d.
\end{equation*}

By Theorem \ref{thm:second_comm} (Part II), we can decompose the inner derivatives into their principal pseudo-differential parts and compact remainders on
$L^p(\R^d)$:
\begin{align*}
    \partial_j [\mathcal{A}_a, M_d] &= \operatorname{Op}\left(\xi_j \sum_{l=1}^d \frac{\partial a}{\partial \xi_l} \frac{\partial d}{\partial x_l}\right) + R_{a,d}^{(j)}, \\
    \partial_j [\mathcal{A}_c, M_b] &= \operatorname{Op}\left(\xi_j \sum_{l=1}^d \frac{\partial c}{\partial \xi_l} \frac{\partial b}{\partial x_l}\right) + R_{c,b}^{(j)}.
\end{align*}

Substituting these decompositions isolates the principal part. The key structural observation is that the amplitude of the principal symbol factorises into
$\xi$- and $x$-dependent pieces, so that no second derivative of $d$ (or of $b$) ever appears:
\begin{equation*}
    \operatorname{Op}\Big(\xi_j \sum_{l=1}^d \tfrac{\partial a}{\partial \xi_l} \tfrac{\partial d}{\partial x_l}\Big) = \sum_{l=1}^d M_{\partial_{x_l}d}\, B_l^a, \qquad B_l^a := \operatorname{Op}(\xi_j \partial_{\xi_l}a),
\end{equation*}
with $B_l^a$ a Fourier multiplier of degree zero (and, symmetrically, $B_l^c := \operatorname{Op}(\xi_j \partial_{\xi_l}c)$).
Writing likewise $\operatorname{Op}(\xi_j\{s_1,s_2\}) = \sum_l M_{b\partial_{x_l}d}\mathcal{A}_c B_l^a - \sum_l M_{d\partial_{x_l}b}\mathcal{A}_a B_l^c$,
the remainder $R_{1,2} = \partial_j[S_1,S_2] - \operatorname{Op}(\xi_j\{s_1,s_2\})$ decomposes exactly as $R_{1,2} = I + II - III$, where
\begin{align*}
    I &= \mathcal{A}_c R_{a,d}^{(j)} M_b - \mathcal{A}_a R_{c,b}^{(j)} M_d, \\
    II &= \sum_{l=1}^d\big( \mathcal{A}_c M_{\partial_{x_l}d} B_l^a M_b - M_{b\partial_{x_l}d}\, \mathcal{A}_c B_l^a \big), \\
    III &= \sum_{l=1}^d\big( \mathcal{A}_a M_{\partial_{x_l}b} B_l^c M_d - M_{d\partial_{x_l}b}\, \mathcal{A}_a B_l^c \big).
\end{align*}

\emph{Term $I$ is compact.} The remainders $R_{a,d}^{(j)}$ and $R_{c,b}^{(j)}$ are compact on $L^p(\R^d)$ by Theorem \ref{thm:second_comm}(II),
while $\mathcal{A}_a, \mathcal{A}_c, M_b, M_d$ are bounded; hence $I$ is compact by the two-sided ideal property of $\mathcal{K}(L^p)$.

\emph{Terms $II$ and $III$ are compact.} Fix $l$ and set $g = \partial_{x_l}d \in C^\epsilon \cap C_0$ (uniformly H\"older since $d \in C^{1,\epsilon}$,
and $C_0$ by Assumption \ref{ass:spatial_reg}). Since $M_{b\partial_{x_l}d} = M_g M_b$, inserting intermediate terms yields the exact identity
\begin{equation*}
    \mathcal{A}_c M_g B_l^a M_b - M_g M_b\, \mathcal{A}_c B_l^a = \underbrace{[\mathcal{A}_c, M_g]\, B_l^a M_b}_{(\ast)} + \underbrace{M_g\, [\mathcal{A}_c B_l^a,\, M_b]}_{(\ast\ast)}.
\end{equation*}
In $(\ast)$, the commutator $[\mathcal{A}_c, M_g]$ is compact on $L^p$ by Lemma \ref{lem:building}(2) (here $g \in C^\epsilon \cap C_0$ and
$c \in C^\infty(\Sp^{d-1})$), and $B_l^a M_b$ is bounded, so $(\ast)$ is compact. In $(\ast\ast)$, the operator $\mathcal{A}_c B_l^a$ is a
Fourier multiplier with smooth symbol $c(\xi)\xi_j\partial_{\xi_l}a(\xi)$; by Theorem \ref{thm:second_comm}(I) together with the zero-order
bound of Theorem \ref{thm:inhomogeneous} (using $b \in C^{1,\epsilon} \cap L^\infty$), the commutator $[\mathcal{A}_c B_l^a, M_b]$ maps
$L^p(\R^d)$ boundedly into $W^{1,p}(\R^d)$, so left multiplication by $g = \partial_{x_l}d \in C_0$ makes $(\ast\ast)$ compact by Lemma
\ref{lem:building}(1) (with $\sigma = 1$). Each $l$-summand of $II$ is therefore compact, and $III$ is handled identically after exchanging
$(a,b) \leftrightarrow (c,d)$.

Being a finite sum of compact operators, $R_{1,2}$ is compact on $L^p(\R^d)$.

Finally, we evaluate the principal symbol. Factoring out $\xi_j$, the algebraic reduction of the main operator matches the definition of the Poisson bracket:
\begin{equation*}
    \xi_j \sum_{l=1}^d \left( \frac{\partial (ab)}{\partial \xi_l} \frac{\partial (cd)}{\partial x_l} - \frac{\partial (ab)}{\partial x_l} \frac{\partial (cd)}{\partial \xi_l} \right) = \xi_j \{s_1, s_2\}.
\end{equation*}

This confirms that $\partial_j [S_1, S_2] = \operatorname{Op}(\xi_j \{s_1, s_2\}) + R_{1,2}$, proving the theorem.
\end{proof}

\section{Comparison with the Classical $L^2$ Theory}\label{sec:comparison}

To place the preceding proof in context, we contrast the global $L^p$ extension with the classical $L^2$ methods developed by Tartar \cite{Tartar1990}
and later used by Antoni\'c and Lazar \cite{AntonicLazar2013}. The transition from Hilbert to Banach spaces breaks the algebraic symmetry of the problem,
requiring a shift from direct integral estimation to singular integral theory to construct H-distributions.

\subsection{The Regularity Threshold}
In the $L^2$ framework, Tartar assumes minimal regularity: the Fourier symbol $a(\xi)$ only requires bounded continuity (e.g., $a \in C^1$ or Lipschitz)
and the spatial multiplier $b$ merely requires an $L^1$-integrable Fourier transform of its gradient ($b \in X^1$). This is possible because of the
Plancherel theorem: an operator is bounded on $L^2$ if and only if its Fourier symbol is in $L^\infty$. 

In the $L^p$ regime ($p \neq 2$), the Plancherel isometry is lost, and boundedness of $\mathcal{A}_a$ instead rests on the H\"ormander--Mikhlin multiplier
theorem, which requires only the fractional Sobolev regularity $a \in W^{s_0,2}(\Sp^{d-1})$ with $s_0 > \frac{d}{2}$. Our elementary proof demands more:
the pointwise Calder\'on--Zygmund kernel bounds on which it relies force the stronger threshold $s > d+2$ of Assumption \ref{ass:symbol}, a distinction
detailed in Remark \ref{rem:sobolev}.

Furthermore, the classical minimal regularity for $b$ does not provide sufficient control for multilinear operators in $L^p$. The Coifman-Meyer
pseudo-differential framework requires the H\"older continuity ($C^{1,\epsilon}$) specified in Assumption \ref{ass:spatial_reg} to establish fractional smoothing.

\subsection{The Mechanism of Compactness}
The main divergence lies in the proof of compactness for the remainder operator $R_j$. In the classical $L^2$ theory the compactness is secured by different
technology rather than by localisation: Tartar's commutation lemmas are themselves global statements on $L^2(\R^d)$, and the escape of mass to infinity
is controlled by decay hypotheses on the multiplier (e.g.\ $b \in C_0$, or $\mathcal{F}(\nabla b) \in L^1$, which forces $\nabla b \in C_0$) combined with
Plancherel-based compactness criteria on the frequency side, where the commutator becomes an integral operator with kernel
$(a(\xi)-a(\eta))\widehat{b}(\xi-\eta)$ amenable to Hilbert-space approximation. What the Hilbert setting supplies is thus a compactness tool
unavailable in $L^p$, not a reduction of the problem to a bounded region. (The mechanism is not Hilbert--Schmidt: for $d \ge 2$
the commutator kernel $k_a(x-y)(b(y)-b(x)) \sim |x-y|^{-d+1}$, and the remainder kernel $\sim |x-y|^{-d+\epsilon}$, fail to be square-integrable near
the diagonal.)

The global $L^p$ setting has neither. Boundedness of $\mathcal{A}_a$ rests on multiplier theorems rather than on Plancherel, and the compactness of $R_j$
must be established directly on the whole of $\R^d$. Our Lemma \ref{lem:remainder} establishes that $R_j$ acts as a smoothing operator of order $-\epsilon$,
mapping into $W^{\epsilon',p}(\R^d)$ for every $\epsilon' < \epsilon$; but the Rellich--Kondrachov embedding fails on the unbounded domain $\R^d$,
since mass may escape to spatial infinity, so local smoothing alone does not suffice.

\subsection{The Global Domain Resolution}
The classical $L^2$ theory already imposes decay and regularity on the multiplier (for instance that $\nabla b$ possess an integrable Fourier
transform) and controls the escape of mass to spatial infinity through those decay hypotheses combined with Plancherel-based compactness,
rather than at the kernel level. Our global formulation must instead confront it directly. To recover compactness of $R_j$ on all of $\R^d$ we
impose only the sup-norm decay $\nabla b \in C_0(\R^d)$ (its gradient remaining uniformly H\"older) and, by constructing spatial truncations
$\chi_n$ and controlling $\nabla b_n - \nabla b$ in the supremum norm, force the non-compact tail of the operator to vanish uniformly.
Global compactness is thus recovered by exhibiting $R_j$ as the uniform-operator limit of locally compact truncations, the Banach-space
substitute for the compactness that Hilbert-space geometry supplies more directly. This global statement is strictly stronger than a local
one (the decay hypothesis $\nabla b \in C_0$ is precisely what Steps 2--3 pay for) and is of independent interest. The transport applications of Sections
\ref{sec:applications} and \ref{sec:wave_eq} do not, however, require this strength: every pairing
there carries a compactly supported factor in each slot, so one may replace each coefficient $b_j(t,\cdot), \rho, c$ by $\chi\,b_j(t,\cdot)$
etc.\ (with $\chi \in C_c^\infty$ equal to $1$ on a neighbourhood of the relevant $x$-support) without any error. Indeed, writing
$\beta := (1-\chi)b$, both $\beta$ and $\nabla\beta$ vanish on that neighbourhood; since the applications use the differentiated commutator
$\partial_{x_j}[\mathcal{A}_a, M_b]$ rather than $[\mathcal{A}_a, M_b]$ itself, this is what must be checked. Tested between slots $\varphi u_n, \varphi v_n$
whose $x$-supports lie where $\chi\equiv 1$, one has $[\mathcal{A}_a, M_\beta](\varphi u_n) = -\beta\,\mathcal{A}_a(\varphi u_n)$ (the term
$\mathcal{A}_a(\beta\varphi u_n)$ vanishing as $\beta\varphi u_n\equiv 0$), whence
$\partial_{x_j}[\mathcal{A}_a, M_\beta](\varphi u_n) = -(\partial_{x_j}\beta)\,\mathcal{A}_a(\varphi u_n) - \beta\,\partial_{x_j}\mathcal{A}_a(\varphi u_n)$;
both summands carry a factor $\beta$ or $\partial_{x_j}\beta$ that vanishes on $\operatorname{supp}(\varphi v_n)$, so the paired integral is exactly zero.
The applications thus consume only the compact-support case (Step 1).
The global lemma becomes indispensable elsewhere, notably in the fractional programme of Section \ref{sec:future}, where $[(-\Delta)^s, M_\varphi]$
carries kernel tails over all of $\R^d$ and exact localisation fails.

\section{Application I: $L^p$ Transport Equations for First-Order Scalar PDEs}\label{sec:applications}

The classical derivation of transport equations for H-measures relies on the Hilbert space geometry of $L^2$. By multiplying the governing equation by the
complex conjugate (or its time derivative) and using the chain rule, Tartar obtained symmetric algebraic energy identities. In the $L^p$ setting ($p \neq 2$),
this algebraic symmetry fractures. We recover a bicharacteristic flow by pairing the primary sequence with its canonical Nemyckij dual
$\Phi_p(u) = |u|^{p-2}u$ \cite{AntonicErcegMisur2021}. Since $\Phi_p$ maps $L_{loc}^p$ boundedly into $L_{loc}^{p'}$, this pairing generates a well-defined
H-distribution; the one subtlety, that $\Phi_p(u_n)$ need not itself converge weakly to zero, is resolved by the following lemma.

\begin{lemma}[Canonical H-distribution]\label{lem:canonical}
Let $1 < p < \infty$ and let $w_n \rightharpoonup 0$ in $L^p_{loc}(\R^d)$; set $z_n := \Phi_p(w_n) = |w_n|^{p-2}w_n$. Then:
\begin{enumerate}
    \item $\{z_n\}$ is bounded in $L^{p'}_{loc}(\R^d)$ (indeed $\int_K |z_n|^{p'} = \int_K |w_n|^p$ for every compact $K$), so along a subsequence
    $z_n \rightharpoonup \bar z$ in $L^{p'}_{loc}$ for some $\bar z \in L^{p'}_{loc}$, and $\tilde z_n := z_n - \bar z \rightharpoonup 0$.
    \item The pair $(w_n, \tilde z_n)$ generates an H-distribution $\mu \in \mathcal{D}'(\R^d \times \Sp^{d-1})$ in the sense of Theorem \ref{thm:h_dist_existence}.
    \item For every $A_n \rightharpoonup 0$ in $L^p_{loc}$, all $\varphi_1, \varphi_2 \in C_c^\infty(\R^d)$, and every $a \in C^\infty(\Sp^{d-1})$,
    \begin{equation*}
        \lim_{n} \int \varphi_1 A_n\, \mathcal{A}_a(\varphi_2 z_n) = \lim_{n} \int \varphi_1 A_n\, \mathcal{A}_a(\varphi_2 \tilde z_n),
    \end{equation*}
    since $\mathcal{A}_a(\varphi_2 \bar z) \in L^{p'}$ is fixed while $\varphi_1 A_n \rightharpoonup 0$ in $L^p$.
\end{enumerate}
Consequently ``the H-distribution generated by the canonical pair $(w_n, \Phi_p(w_n))$'' is unambiguously defined as $\mu$; taking $A_n = w_n$,
$\varphi_1 = \varphi_2 = \varphi$, $a = 1$ gives the non-negative diagonal $\langle \mu, \varphi^2 \rangle = \lim_n \int |w_n|^p \varphi^2$, the
$L^p$ concentration defect. The same conclusions hold componentwise for vector sequences $\mathbf{w}_n \rightharpoonup 0$ in $L^p_{loc}(\R^d; \R^m)$
paired with $\Phi_p(\mathbf{w}_n)$, a single finite diagonal subsequence serving all $m$ components at once. The applications of Sections
\ref{sec:applications} and \ref{sec:wave_eq} require the analogue in which the sequences additionally carry an evolution parameter $t$
while $\mathcal{A}_a$ still acts in $x$ alone; that parametrised statement is more delicate: it fails for arbitrary weakly-null
sequences (Remark \ref{rem:inertcounterexample}), and is established, for the solution sequences to which we apply it, in Lemma \ref{lem:evolution}.
\end{lemma}

\begin{proof}
For (1), $\int_K |z_n|^{p'} = \int_K |w_n|^{(p-1)p'} = \int_K |w_n|^p \le \norm{w_n}_{L^p(K)}^p$ is bounded, and reflexivity of $L^{p'}_{loc}$ yields a
weakly convergent subsequence with limit $\bar z$. Claim (2) is immediate from Theorem \ref{thm:h_dist_existence}, both $w_n$ and $\tilde z_n$ being weakly
null. For (3), the difference of the two integrals equals $\int \varphi_1 A_n\, \mathcal{A}_a(\varphi_2 \bar z)$;
the function $g := \mathcal{A}_a(\varphi_2 \bar z)$ lies in $L^{p'}(\R^d)$ because $\mathcal{A}_a$ is bounded on $L^{p'}$ and $\varphi_2 \bar z \in L^{p'}$
has compact support, while $\varphi_1 A_n \rightharpoonup 0$ in $L^p$, so $\int (\varphi_1 A_n)\, g \to 0$.
\end{proof}

\begin{remark}[Failure of the naive inert-parameter construction]\label{rem:inertcounterexample}
One might expect Lemma \ref{lem:canonical} to extend verbatim to sequences on $\R_t\times\R^d_x$ with $t$ an inert parameter and $\mathcal{A}_a$
acting in $x$ alone. This is false for general weakly-null sequences. Fix $1<p<\infty$ and set, on $\R_t\times\R^d_x$,
\begin{equation*}
    w_n(t,x) := \cos(nt), \qquad z_n := \Phi_p(w_n) = |\cos nt|^{p-2}\cos(nt),
\end{equation*}
both independent of $x$. By the Riemann--Lebesgue lemma $w_n\rightharpoonup 0$ in $L^p_{loc}$ and $z_n\rightharpoonup 0$ in $L^{p'}_{loc}$;
the periodic mean of $z_n$ vanishes, so $\bar z=0$ and $\tilde z_n=z_n$, a bona fide canonical pair. Since $\mathcal{A}_a$ acts only in $x$
it leaves the $x$-independent factor $z_n$ untouched, and $w_n z_n=|\cos nt|^p\rightharpoonup c_p:=\frac1{2\pi}\int_0^{2\pi}|\cos\tau|^p\,\rd\tau>0$.
Hence, for $\varphi_i(t,x)=\chi(t)\psi_i(x)$,
\begin{equation*}
    \lim_n\int_{\R\times\R^d}\varphi_1 w_n\,\mathcal{A}_a(\varphi_2 z_n)\,\rd x\,\rd t
    = c_p\int_\R\chi^2\,\rd t\;\int_{\R^d}\psi_1\,\mathcal{A}_a\psi_2\,\rd x.
\end{equation*}
Were this limit equal to $\langle\mu,\varphi_1\varphi_2\,a\rangle$ for a distribution $\mu$ on $(\R\times\R^d)\times\Sp^{d-1}$, then taking $\psi_1,\psi_2$
with disjoint supports, so that $\varphi_1\varphi_2\equiv 0$, would force the right-hand side to vanish. But
$\int\psi_1\mathcal{A}_a\psi_2=\iint\psi_1(x)\,k_a(x-y)\,\psi_2(y)\,\rd y\,\rd x$, with $k_a$ the off-diagonal convolution kernel of $\mathcal{A}_a$,
is nonzero for suitable disjoint bumps unless $a$ is constant. No such $\mu$ exists: pure temporal oscillations are invisible to the $x$-multiplier yet
correlate the canonical pair. The obstruction is structural: $[\mathcal{A}_a,M_{\psi_1}]$ commutes with multiplication by $e^{int}$, hence is never
compact on $L^{p'}(\R\times\R^d)$, so the factorisation underpinning Theorem \ref{thm:h_dist_existence} breaks down, and is removed only by invoking
the evolution equation. Indeed $w_n$ solves no transport equation with $L^p_{loc}$ datum, since $\partial_t w_n=-n\sin(nt)$ is unbounded; equivalently
the slice $t\mapsto\langle w_n(t,\cdot),\theta\rangle=\cos(nt)\int\theta$ is not equicontinuous, precisely the hypothesis \textup{(ii)} of
Lemma \ref{lem:evolution} that solution sequences do satisfy.
\end{remark}

\begin{lemma}[Evolution H-distribution]\label{lem:evolution}
Let $1<p<\infty$, let $I\subset\R$ be a bounded open interval and $\Omega\subset\R^d$ open. Suppose $u_n\rightharpoonup 0$ in $L^p_{loc}(I\times\Omega)$
and $v_n\rightharpoonup 0$ in $L^{p'}_{loc}(I\times\Omega)$ satisfy, for every $J\Subset I$ and every compact $K\subset\Omega$:
\begin{enumerate}
    \item[\textup{(i)}] \emph{uniform time-slices:} $\displaystyle\sup_n\operatorname*{ess\,sup}_{t\in J}\big(\norm{u_n(t,\cdot)}_{L^p(K)}+\norm{v_n(t,\cdot)}_{L^{p'}(K)}\big)<\infty$;
    \item[\textup{(ii)}] \emph{equicontinuous slices:} for each $\theta\in C_c^\infty(\Omega)$ the maps $t\mapsto\langle u_n(t,\cdot),\theta\rangle$ and $t\mapsto\langle v_n(t,\cdot),\theta\rangle$ admit continuous representatives that are equicontinuous on $J$, uniformly in $n$.
\end{enumerate}
Then, along a subsequence, there is a distribution $\mu\in\mathcal{D}'(I\times\Omega\times\Sp^{d-1})$, of order $0$ in $(t,x)$ and of finite order in the
sphere variable (the pairing being continuous in $\norm{a}_{C^\kappa(\Sp^{d-1})}$ with $\kappa=\lfloor d/2\rfloor+1$ on tensor products, so that by the
kernel theorem for distributions of anisotropic order the anisotropic order of $\mu$ on the product is $\le d(\kappa+2)$ in $\xi$ \cite{AntonicErcegMisur2021}),
such that for all $\varphi_1,\varphi_2\in C_c^\infty(I\times\Omega)$ and $a\in C^\infty(\Sp^{d-1})$,
\begin{equation*}
    \lim_n\int_{I\times\Omega}\varphi_1\,u_n\,\mathcal{A}_a(\varphi_2 v_n)\,\rd x\,\rd t=\langle\mu,\varphi_1\varphi_2\,a\rangle,
\end{equation*}
where $\mathcal{A}_a$ acts in the $x$-variable alone. Exhausting $\R\times\R^d$ by such $I\times\Omega$ extends $\mu$ to
$\mathcal{D}'(\R\times\R^d\times\Sp^{d-1})$. Hypotheses \textup{(i)}--\textup{(ii)} hold for the solution sequences of Theorems
\textup{\ref{thm:bicharacteristic_first_order}}, \textup{\ref{thm:bicharacteristic_pwave}}, and \textup{\ref{thm:wave-variable}},
as verified in their proofs. The standing weak-nullity $v_n\rightharpoonup 0$ holds outright in the wave cases (Theorems
\textup{\ref{thm:bicharacteristic_pwave}}, \textup{\ref{thm:wave-variable}}), where $\Phi_2=\operatorname{id}$ and $v_n=w_n\rightharpoonup 0$;
in the scalar case $p\neq 2$ (Theorem \textup{\ref{thm:bicharacteristic_first_order}}) the canonical dual $\Phi_p(u_n)$ need not be weakly null,
and the lemma is applied to the shifted dual $\tilde v_n:=\Phi_p(u_n)-\bar v$ ($\bar v$ its weak $L^{p'}_{loc}$ limit along a subsequence),
the pairing being insensitive to the shift by Lemma \ref{lem:canonical}(3).
\end{lemma}

\begin{proof}
Fix $J\Subset I$ and $K\subset\Omega$ compact; constants are uniform in $n$.

\emph{Slice representatives.} By (i) each $u_n\in L^\infty(J;L^p(K))$ with $\sup_n\norm{u_n}_{L^\infty(J;L^p(K))}<\infty$, and by (ii) the maps
$t\mapsto\langle u_n(t,\cdot),\theta\rangle$ are continuous for every $\theta$ in a countable $L^{p'}(K)$-dense family $\mathcal{D}$. We construct, for each $n$,
a representative of $u_n$ defined for every $t\in J$, weakly continuous as a map $J\to L^p(K)$, and satisfying $\norm{u_n(t,\cdot)}_{L^p(K)}\le C$
for all $t\in J$. Fix $t_0\in J$ and let $t_k\to t_0$ range over good points (a.e.\ $t$ is good, the essential slice bound of (i) holding there); the
sequence $u_n(t_k,\cdot)$ is bounded in the reflexive space $L^p(K)$, hence weakly precompact, and every weak subsequential limit $\ell$ satisfies
$\langle\ell,\theta\rangle=\lim_k\langle u_n(t_k,\cdot),\theta\rangle$ for each $\theta\in\mathcal{D}$, by the continuity of the slice maps (ii), a value
independent of the points $t_k$ and of the subsequence. Since $\mathcal{D}$ is $L^{p'}(K)$-dense and all these limits obey the uniform bound $\le C$, the
functionals $\theta\mapsto\langle\ell,\theta\rangle$ agree on $\mathcal{D}$ and extend continuously, so $\ell$ is the same for all admissible choices; the
whole bounded net $\{u_n(t,\cdot): t\to t_0,\ t\text{ good}\}$ thus converges weakly in $L^p(K)$ to a unique element, which we take as $u_n(t_0,\cdot)$. Weak
lower semicontinuity gives $\norm{u_n(t_0,\cdot)}_{L^p(K)}\le C$, and continuity of $t\mapsto\langle u_n(t,\cdot),\theta\rangle$ for $\theta\in\mathcal{D}$
together with the uniform bound yields weak continuity of $u_n(\cdot,\cdot):J\to L^p(K)$.
The same construction applies to $v_n$ in $L^{p'}(K)$. We work with these representatives throughout, so every slice statement below holds for every $t\in J$.

\emph{Step 1 (fiberwise weak nullity).} For $\theta\in C_c^\infty(\Omega)$ set $g_n(t):=\langle u_n(t,\cdot),\theta\rangle$.
By (i), $|g_n|\le C\norm{\theta}_{L^{p'}}$ on $J$; by (ii), $\{g_n\}$ is equicontinuous. Arzel\`a--Ascoli extracts a subsequence
converging uniformly on $J$; and $\int_J\rho\,g_n=\int u_n\,(\rho\otimes\theta)\to 0$ for every $\rho\in C_c(J)$ because $u_n\rightharpoonup 0$,
so every uniform limit point vanishes and hence $g_n(t)\to 0$ for every $t\in J$. Ranging $\theta$ over a countable $L^{p'}(K)$-dense set
and using the pointwise slice bound just established, $u_n(t,\cdot)\rightharpoonup 0$ in $L^p(K)$ for every $t\in J$;
likewise $v_n(t,\cdot)\rightharpoonup 0$ in $L^{p'}(K)$.

\emph{Step 2 (a bounded, symmetrically factorising form).} For $\varphi_1,\varphi_2\in C_c^\infty(I\times\Omega)$ and $a\in C^\infty(\Sp^{d-1})$
put $L_n:=\int\varphi_1 u_n\,\mathcal{A}_a(\varphi_2 v_n)$. By (i) and the $L^{p'}$-boundedness of $\mathcal{A}_a$ (H\"ormander--Mikhlin),
$|L_n|\le C\norm{a}_{C^\kappa}$ with $\kappa=\lfloor d/2\rfloor+1$, uniformly in $n$; a diagonal subsequence makes $L_n$ converge for all
$(\varphi_1,\varphi_2,a)$ in a countable dense family, and the uniform bound extends the limit to a trilinear form $B$ with
$|B|\le C\norm{\varphi_1}_{C_0}\norm{\varphi_2}_{C_0}\norm{a}_{C^\kappa}$. We claim the symmetric factorisation
\begin{equation*}
    B(\eta\varphi_1,\varphi_2;a)=B(\varphi_1,\eta\varphi_2;a),\qquad\eta\in C_c^\infty(I\times\Omega),
\end{equation*}
for every space--time cutoff $\eta$; its difference equals $-\lim_n\int\varphi_1 u_n\,[\mathcal{A}_a,M_\eta](\varphi_2 v_n)$, and we treat $\eta$
in three stages. \emph{(a) Purely spatial $\eta=\eta(x)$:} for each fixed $t$, $[\mathcal{A}_a,M_\eta]$ is compact on $L^{p'}(\R^d)$
(Lemma \ref{lem:building}(2), with $\eta\in C_c^\infty(\R^d)\subset C^\epsilon\cap C_0$ and $a\in C^\infty(\Sp^{d-1})$), so it carries the
fiberwise weakly-null $\varphi_2(t,\cdot)v_n(t,\cdot)$ to a strongly $L^{p'}$-null sequence; paired with $\varphi_1(t,\cdot)u_n(t,\cdot)$
(of compact $x$-support, bounded in $L^p$), the $x$-integral tends to $0$ for every $t$ and is dominated by
$C\operatorname*{ess\,sup}_t\norm{u_n(t)}_{L^p(K)}\operatorname*{ess\,sup}_t\norm{v_n(t)}_{L^{p'}(K)}\mathbf 1_J(t)\in L^1$, so dominated convergence
gives the identity. \emph{(b) Purely temporal $\eta=\eta(t)$:} here $[\mathcal{A}_a,M_{\eta(t)}]=0$ identically, since $\mathcal{A}_a$ acts in $x$
alone and multiplication by the $x$-independent factor $\eta(t)$ commutes with it, so the difference vanishes exactly. \emph{(c) General $\eta$:}
for a tensor product $\eta=\alpha(t)\beta(x)$ with $\alpha\in C_c^\infty(I)$ and $\beta\in C_c^\infty(\Omega)$, composing (a) with $\beta$ and (b)
with $\alpha$ gives $B(\alpha\beta\varphi_1,\varphi_2;a)=B(\alpha\varphi_1,\beta\varphi_2;a)=B(\varphi_1,\alpha\beta\varphi_2;a)$.
By the Stone--Weierstrass theorem, finite sums of such tensor products approximate any $\eta\in C_c^\infty(I\times\Omega)$ uniformly on the fixed
compact $\operatorname{supp}\varphi_1\cup\operatorname{supp}\varphi_2$; the approximants may be taken with $C_c^\infty$ factors: fix once and
for all a product cutoff $\zeta=\zeta_1(t)\zeta_2(x)$ with $\zeta_1\in C_c^\infty(I)$, $\zeta_2\in C_c^\infty(\Omega)$ and $\zeta\equiv 1$ on a
neighbourhood of $\operatorname{supp}\varphi_1\cup\operatorname{supp}\varphi_2$, and replace each tensor sum $\sum\alpha\beta$ by
$\zeta\sum\alpha\beta=\sum(\zeta_1\alpha)(\zeta_2\beta)$, still a finite tensor sum whose factors $\zeta_1\alpha\in C_c^\infty(I)$,
$\zeta_2\beta\in C_c^\infty(\Omega)$ satisfy the hypotheses of (a) and (b) and which agrees with the original sum on
$\operatorname{supp}\varphi_1\cup\operatorname{supp}\varphi_2$, so that the factorisation just established applies to each summand; since
$|B(\psi_1,\psi_2;a)|\le C\norm{\psi_1}_{C_0}\norm{\psi_2}_{C_0}\norm{a}_{C^\kappa}$ makes both $\eta\mapsto B(\eta\varphi_1,\varphi_2;a)$
and $\eta\mapsto B(\varphi_1,\eta\varphi_2;a)$ continuous for the supremum norm, the identity extends to every $\eta\in C_c^\infty(I\times\Omega)$.

\emph{Step 3 (the distribution).} The symmetric factorisation says precisely that $B(\varphi_1,\varphi_2;a)$ depends on $(\varphi_1,\varphi_2)$
only through the product $\varphi_1\varphi_2$: for $\psi\in C_c^\infty(I\times\Omega)$ set $\langle T_a,\psi\rangle:=B(\psi,\zeta;a)$ with any
$\zeta\in C_c^\infty(I\times\Omega)$ equal to $1$ on a neighbourhood of $\operatorname{supp}\psi$; this is independent of the choice of $\zeta$
(if $\zeta,\zeta'$ both equal $1$ near $\operatorname{supp}\psi$, pick $\zeta''\equiv1$ on $\operatorname{supp}\psi$ with
$\operatorname{supp}\zeta''\subset\{\zeta=\zeta'\}$; the shuffling identity gives
$B(\psi,\zeta-\zeta';a)=B(\zeta''\psi,\zeta-\zeta';a)=B(\psi,\zeta''(\zeta-\zeta');a)=0$), linear in $\psi$, and bounded by
$|\langle T_a,\psi\rangle|\le C\norm{\psi}_{C^0}\norm{a}_{C^\kappa}$. The resulting bilinear map $(\psi,a)\mapsto\langle T_a,\psi\rangle$ on
$C_c^\infty(I\times\Omega)\times C^\infty(\Sp^{d-1})$, continuous of order $0$ in $(t,x)$ and, on tensor products, order $\kappa$ on the sphere, is
exactly the input of the Schwartz-kernel argument for distributions of anisotropic order that establishes Theorem \ref{thm:h_dist_existence}
\cite{AntonicMitrovic2011, AntonicErcegMisur2021}; there the base $X = I\times\Omega$ may be any $C^\infty$ manifold (the multiplier $\mathcal{A}_a$
acting in $x$ alone enters only through the continuity of $a\mapsto T_a$, not through the kernel construction), so it produces a unique
$\mu\in\mathcal{D}'(I\times\Omega\times\Sp^{d-1})$ with $B(\varphi_1,\varphi_2;a)=\langle\mu,\varphi_1\varphi_2\,a\rangle$, of order $0$ in $(t,x)$. The passage
from tensor-product continuity to the full product raises the sphere order to $\le d(\kappa+2)$ \cite[Rem.~2]{AntonicErcegMisur2021}, which is immaterial
below, every symbol paired against $\mu$ being smooth.
Exhausting $\R\times\R^d$ by slabs $I\times\Omega$ and diagonalising extends $\mu$ to $\mathcal{D}'(\R\times\R^d\times\Sp^{d-1})$.
\end{proof}

\begin{remark}[Why the gradients are not bounded uniformly]\label{rem:novacuity}
In Theorems \ref{thm:bicharacteristic_first_order} and \ref{thm:bicharacteristic_pwave} the regularity $u_n\in W^{1,p}_{loc}$ (resp.\ $W_n\in W^{1,2}_{loc}$)
is imposed on each term individually (to license the chain rule for $\Phi_p$ and the pointwise identities below), but no uniform bound on the
gradients is assumed, and none may be: were $\{u_n\}$ bounded in $W^{1,p}_{loc}$, Rellich--Kondrachov would upgrade $u_n\rightharpoonup 0$ to $u_n\to 0$
strongly in $L^p_{loc}$, forcing $\mu\equiv 0$ and rendering the transport equation vacuous. The two hypotheses are perfectly compatible: for each fixed
$n$ the gradient $\nabla u_n$ is a bona fide $L^p_{loc}$ function of finite norm (as $W^{1,p}_{loc}$ demands), and it is only the growth of these norms
$\norm{\nabla u_n}_{L^p(K)}\to\infty$ as $n\to\infty$ that is left unconstrained. The oscillatory sequences carrying nontrivial microlocal
energy ($u_n(t,x)=U(n\Phi(t,x))$ with $U$ periodic) have $\norm{\nabla u_n}\sim n$ and are excluded by any uniform gradient bound. What the
construction consumes is only the uniform $L^\infty_{loc}$ bound (scalar case) or the uniform local-energy bound (wave case), hypothesis \textup{(i)} of
Lemma \ref{lem:evolution}, together with the equicontinuity \textup{(ii)} that the evolution equation supplies.
\end{remark}

\begin{theorem}[Bicharacteristic Flow in $L^p$]\label{thm:bicharacteristic_first_order}
Let $2 \le p < \infty$. Let $u_n \rightharpoonup 0$ in $L_{loc}^p(\R \times \R^d)$ be a sequence of real-valued solutions, each of class
$W^{1,p}_{loc} \cap L^\infty_{loc}(\R \times \R^d)$ and uniformly bounded in $L^\infty_{loc}$ (no uniform bound on the gradients being imposed;
cf.\ Remark \ref{rem:novacuity}), to the first-order scalar equation:
\begin{equation*}
    \partial_t u_n + \sum_{j=1}^d b_j(t,x) \partial_{x_j} u_n + c(t,x) u_n = f_n \to 0 \quad \text{in } L_{loc}^p(\R \times \R^d),
\end{equation*}
where $b_j \in C^{1,\epsilon}(\R \times \R^d)$ with $\nabla_x b_j(t,\cdot) \in C_0(\R^d; \R^d)$ uniformly in $t$ (sup-norm decay, uniformly H\"older;
so that the parametrised Second Commutation Lemma of Remark \ref{rem:parameter} applies), and $c \in C_b(\R \times \R^d)$ is bounded and continuous.
Let $\mu$ be the H-distribution generated by the canonical pair $(u_n, \Phi_p(u_n))$ in the sense of Lemmas \ref{lem:canonical} and \ref{lem:evolution}.
Let $\dot{x} = b(t,x)$ and $\dot{\xi}_k = -\sum_{j=1}^d \xi_j \partial_{x_k} b_j$ define the bicharacteristic flow on the cotangent bundle, and let
$\dot{\xi}^\top = \dot{\xi} - (\dot{\xi}\cdot\xi)\xi$ denote its tangential projection onto $T_\xi\Sp^{d-1}$. Then $\mu$ satisfies the phase-space
transport equation:
\begin{equation*}
    \partial_t \mu + b \cdot \nabla_x \mu + \operatorname{div}_{\Sp^{d-1}}(\dot{\xi}^\top\mu) + p c \mu = 0
\end{equation*}
in the sense of distributions on $\R \times \R^d \times \Sp^{d-1}$, where $\operatorname{div}_{\Sp^{d-1}}$ denotes the intrinsic divergence on the unit sphere.
\end{theorem}

\begin{remark}[The range $1 < p < 2$]\label{rem:chainrange}
For $2 \le p < \infty$ the Nemyckij map $\Phi_p$ is $C^1$ with $\Phi_p'(s) = (p-1)|s|^{p-2}$ bounded on bounded sets, so under the standing hypothesis
$u_n \in L^\infty_{loc} \cap W^{1,p}_{loc}$ one has $v_n = \Phi_p(u_n) \in W^{1,p'}_{loc}$ with $\nabla v_n = (p-1)|u_n|^{p-2}\nabla u_n$; here $p' \le 2 \le p$,
so $\nabla u_n \in L^{p'}_{loc}$ locally. For $1 < p < 2$ the derivative $\Phi_p'$ is singular at the origin: near a transversal zero of $u_n$ one has
$|u_n|^{p-2}\nabla u_n \in L^{p'}_{loc}$ only when $p > \tfrac{1+\sqrt{5}}{2}$, so the dual sequence $v_n$ need not lie in $W^{1,p'}_{loc}$ and the derivation
below does not extend to $1 < p < 2$ without further non-degeneracy hypotheses. We therefore restrict Theorem \ref{thm:bicharacteristic_first_order} to $p \ge 2$.
\end{remark}

\begin{proof}
Throughout, we take the symbol $a \in C^\infty(\Sp^{d-1})$ real-valued (the general case follows by treating real and imaginary parts separately);
and we pass without relabelling to the subsequence along which the H-distribution $\mu$ of the canonical pair $(u_n, v_n)$, $v_n := \Phi_p(u_n)$, exists.
Since $\Phi_p(u_n)$ need not be weakly null, we first extract (reflexivity of $L^{p'}_{loc}$) a subsequence with $v_n \rightharpoonup \bar v$ and pass to
the weakly-null shifted dual $\tilde v_n := v_n - \bar v \rightharpoonup 0$; by Lemma \ref{lem:canonical}(3) (whose evolution form is identical,
the difference $\int \varphi_1 u_n\,\mathcal{A}_a(\varphi_2 \bar v)\to 0$ since $\varphi_1 u_n \rightharpoonup 0$ pairs against the fixed $L^{p'}$ function
$\mathcal{A}_a(\varphi_2 \bar v)$) every pairing against $v_n$ equals that against $\tilde v_n$, so $\mu$ is unambiguously the canonical H-distribution and
may be produced by applying Lemma \ref{lem:evolution} to the weakly-null pair $(u_n, \tilde v_n)$. Its existence is thereby furnished by Lemma \ref{lem:evolution}:
hypothesis (i) holds because the uniform $L^\infty_{loc}$ bound gives
$\sup_n\operatorname*{ess\,sup}_{t\in J}(\norm{u_n(t,\cdot)}_{L^p(K)} + \norm{v_n(t,\cdot)}_{L^{p'}(K)}) < \infty$
on every slab $J\times K$ (recall $|v_n|^{p'} = |u_n|^p$; the shift by $\bar v \in L^\infty(J;L^{p'}(K))$ preserves it),
while hypothesis (ii) holds because, for $\theta \in C_c^\infty(\R^d)$ and $s,t \in J$,
\begin{equation*}
    \langle u_n(t,\cdot) - u_n(s,\cdot), \theta\rangle = \int_s^t \Big[ \big\langle u_n, \textstyle\sum_j \partial_{x_j}(b_j\theta) - c\theta\big\rangle + \langle f_n, \theta\rangle \Big]\,\rd\tau,
\end{equation*}
which by (i), the boundedness of $b_j, c$, and H\"older's inequality in $\tau$ is
$\le C(\theta)\big(|t-s| + \norm{f_n}_{L^p(J\times K)}|t-s|^{1/p'}\big) \le C'(\theta)|t-s|^{1/p'}$, uniformly in $n$ (the numbers $\norm{f_n}_{L^p(J\times K)}$
being bounded, as $f_n\to 0$); the dual equation controls the slices of $v_n$ identically. Finally, the shifted dual $\tilde v_n := v_n - \bar v$ inherits
both hypotheses: (i) is preserved since $\bar v \in L^\infty(J;L^{p'}(K))$ by weak lower semicontinuity, and (ii) holds because
$\langle\tilde v_n(t,\cdot) - \tilde v_n(s,\cdot),\theta\rangle = \langle v_n(t,\cdot) - v_n(s,\cdot),\theta\rangle - \langle\bar v(t,\cdot) - \bar v(s,\cdot),\theta\rangle$,
where $t\mapsto\langle\bar v(t,\cdot),\theta\rangle$ is the uniform-on-$J$ limit (Arzel\`a--Ascoli) of the equicontinuous family
$t\mapsto\langle v_n(t,\cdot),\theta\rangle$ and hence shares their modulus, so the difference is equicontinuous uniformly in $n$.

Since $p \ge 2$, Remark \ref{rem:chainrange} gives $v_n = \Phi_p(u_n) = |u_n|^{p-2}u_n \in W^{1,p'}_{loc}(\R \times \R^d)$
with $\partial_{x_j} v_n = (p-1)|u_n|^{p-2}\partial_{x_j} u_n$. Multiplying the primary equation by $(p-1)|u_n|^{p-2}$
therefore shows that the dual sequence $v_n$ satisfies the dual equation:
\begin{equation*}
    \partial_t v_n + \sum_{j=1}^d b_j(t,x) \partial_{x_j} v_n + (p-1)c(t,x) v_n = f_n^* \to 0 \quad \text{in } L_{loc}^{p'}(\R \times \R^d),
\end{equation*}
where $f_n^* = (p-1)|u_n|^{p-2}f_n$. Let $\varphi \in C_c^1(\R \times \R^d)$ be a smooth, real-valued cutoff function.
Multiplying the primary equation by $\varphi$ yields:
\begin{equation*}
    \partial_t (\varphi u_n) + \sum_{j=1}^d b_j \partial_{x_j} (\varphi u_n) + c \varphi u_n = \varphi f_n + (\partial_t \varphi)u_n + \sum_{j=1}^d b_j (\partial_{x_j} \varphi)u_n.
\end{equation*}

We apply the bounded Fourier multiplier $\mathcal{A}_a$ (with real symbol $a \in C^\infty(\Sp^{d-1})$) to this localized equation.
Since $\mathcal{A}_a$ commutes with $\partial_{x_j}$, one has
$[\mathcal{A}_a, M_{b_j}]\partial_{x_j} = \partial_{x_j}[\mathcal{A}_a, M_{b_j}] - [\mathcal{A}_a, M_{\partial_{x_j}b_j}]$,
and Theorem \ref{thm:second_comm}(II) (the Global $L^p$ Second Commutation Lemma) decomposes
$\partial_{x_j}[\mathcal{A}_a, M_{b_j}] = \operatorname{Op}(\xi_j\{a,b_j\}) + R_j$; therefore
\begin{equation*}
    \mathcal{A}_a(b_j \partial_{x_j} (\varphi u_n)) = b_j \partial_{x_j} \mathcal{A}_a(\varphi u_n) + \operatorname{Op}(\xi_j \{a, b_j\}) (\varphi u_n) + \mathcal{R}_{j,n},
\end{equation*}
where the remainder $\mathcal{R}_{j,n} := R_j(\varphi u_n) - [\mathcal{A}_a, M_{\partial_{x_j}b_j}](\varphi u_n)$ is the image of $\varphi u_n$ under a
compact operator: $R_j$ is compact on $L^p(\R^d)$ by Theorem \ref{thm:second_comm}(II), and $[\mathcal{A}_a, M_{\partial_{x_j}b_j}]$ is compact on
$L^p(\R^d)$ by Lemma \ref{lem:building}(2) (as $\partial_{x_j}b_j \in C^\epsilon\cap C_0$), both with operator norms uniform for $t$ in a bounded interval
(Remark \ref{rem:parameter}).

Substituting this into the localized equation, we pair the result with the localized dual sequence $\varphi v_n$ in the $L^p$-$L^{p'}$ duality pairing.
Simultaneously, we multiply the dual PDE by $\varphi$, apply no Fourier multiplier, and pair it against $\mathcal{A}_a(\varphi u_n)$. Summing these two
bilinear pairings forces the cross-terms to collapse via the Leibniz rule:
\begin{equation*}
    \int \partial_t \mathcal{A}_a(\varphi u_n) (\varphi v_n) + \int \mathcal{A}_a(\varphi u_n) \partial_t (\varphi v_n) = \int \partial_t \left[ \mathcal{A}_a(\varphi u_n) (\varphi v_n) \right] = 0,
\end{equation*}
where the integral vanishes by the compact support of $\varphi$. 

The spatial transport terms undergo a similar algebraic reduction:
\begin{equation*}
    \int \sum_{j=1}^d b_j \partial_{x_j} \mathcal{A}_a(\varphi u_n) (\varphi v_n) + \int \sum_{j=1}^d b_j \mathcal{A}_a(\varphi u_n) \partial_{x_j} (\varphi v_n) = -\int \sum_{j=1}^d (\partial_{x_j}b_j) \mathcal{A}_a(\varphi u_n) (\varphi v_n).
\end{equation*}

\emph{The remainder pairings vanish.} Substituting the commutation identity contributes, to the summed bilinear pairing,
the term $\sum_{j=1}^d\int \mathcal{R}_{j,n}\,(\varphi v_n)$. For each $t$ the operator producing $\mathcal{R}_{j,n}(\cdot\,;t)$
is compact on $L^p(\R^d)$ with $t$-uniform norm, while $(\varphi u_n)(t,\cdot) \rightharpoonup 0$ in $L^p(\R^d)$ by the fiberwise
weak nullity of Lemma \ref{lem:evolution} (Step 1 of its proof); hence $\mathcal{R}_{j,n}(t,\cdot) \to 0$ strongly in $L^p(\R^d)$ for every $t$,
and, paired against the $L^{p'}$-bounded slice $(\varphi v_n)(t,\cdot)$, the $x$-integral tends to $0$ for every $t$ and is dominated by
$C\operatorname*{ess\,sup}_t\norm{(\varphi u_n)(t)}_{L^p}\operatorname*{ess\,sup}_t\norm{(\varphi v_n)(t)}_{L^{p'}}\mathbf{1}_J(t) \in L^1(\rd t)$.
Dominated convergence gives $\sum_j\int\mathcal{R}_{j,n}(\varphi v_n) \to 0$, so these terms drop out of every limit below.

We now pass to the limit as $n \to \infty$. A word on symbol bookkeeping: the H-distribution pairing is bilinear, so moving $\mathcal{A}_a$
from one slot to the other is the transpose $\mathcal{A}_a^\top = \mathcal{A}_{\check a}$ with $\check a(\xi):=a(-\xi)$, not a Hilbert adjoint.
Every identification below therefore produces the reflected symbol $\check a$; but the same reflection occurs in all terms, and as $a$ ranges over
$C^\infty(\Sp^{d-1})$ so does $\check a$, so we relabel $\check a\mapsto a$ uniformly and display each term with the symbol $a$ (this uniform reflection
is the sole role played by the reality of $a$). Thus, by Lemma \ref{lem:canonical} (moving $\mathcal{A}_a$ onto the dual slot, the fixed weak limit of $v_n$
contributing nothing), $\lim \int \mathcal{A}_a(\varphi u_n)(\varphi v_n) = \langle \mu, a \varphi^2 \rangle$. The spatial reduction above converges exactly to
$-\langle \mu, a (\nabla \cdot b) \varphi^2 \rangle$. The zero-order multiplication terms sum algebraically, incorporating the fundamental $L^p$ scaling:
\begin{equation*}
    \int \mathcal{A}_a(c \varphi u_n) (\varphi v_n) + \int (p-1)c(\varphi v_n) \mathcal{A}_a(\varphi u_n) \xrightarrow{n \to \infty} \langle \mu, p a c \varphi^2 \rangle.
\end{equation*}
Here the continuity of $c$ is essential: by the transpose relation just described each term is the integral of $c$ against the $L^1$-bounded density
$\varphi u_n\, \mathcal{A}_a(\varphi v_n)$, whose weak-$*$ limit is a Radon measure and hence pairs with the bounded continuous coefficient $c$
(equivalently, $\langle\mu, ca\varphi^2\rangle$ is defined by uniformly approximating $c \in C_b$ by smooth symbols). The two contributions carry the
coefficients $c$ and $(p-1)c$, which sum to the $L^p$-scaled factor $pc$.

Finally, we treat the principal-symbol term. By the Second Commutation Lemma its amplitude factorises into $\xi$- and $x$-dependent pieces,
$S_j(b_j) = \sum_{l=1}^d M_{\partial_{x_l}b_j}\,\mathcal{A}_{b_{jl}}$, with the fixed smooth degree-zero multiplier
$\mathcal{A}_{b_{jl}} := \operatorname{Op}(\xi_j\partial_{\xi_l}a)$ (here $b_{jl}(\xi) := \xi_j\partial_{\xi_l}a(\xi) \in C^\infty(\Sp^{d-1})$)
and the merely continuous amplitude $\partial_{x_l}b_j \in C_0(\R^d)$; no second derivative of $b_j$ appears. Its contribution
to the pairing is $\sum_{l=1}^d\int (\partial_{x_l}b_j)\,(\varphi v_n)\,\mathcal{A}_{b_{jl}}(\varphi u_n)$. Set $g := \partial_{x_l}b_j$, bounded and
continuous (indeed $g \in C_0(\R^d)$). Transposing $\mathcal{A}_{b_{jl}}$ onto the dual slot turns each summand into
$\int (\varphi u_n)\,\mathcal{A}_{\check b_{jl}}(g\,\varphi v_n)$; its limit is identified exactly as for the second zero-order contribution above, by
uniformly approximating $g$ by smooth symbols. For $g_k \in C_c^\infty$ with $g_k \to g$ in $L^\infty$ on the compact $x$-support, the factor $g_k\varphi$
is a legitimate test function, so $\int (\varphi u_n)\,\mathcal{A}_{\check b_{jl}}(g_k\varphi v_n) \to \langle\mu, g_k\,\varphi^2\,\check b_{jl}\rangle$ by the
definition of $\mu$, while the error $\int(\varphi u_n)\mathcal{A}_{\check b_{jl}}((g-g_k)\varphi v_n)$ is $\le C\norm{g-g_k}_{L^\infty}\to 0$ uniformly in $n$
(by (i) and the $L^{p'}$-boundedness of $\mathcal{A}_{\check b_{jl}}$). Since the H-distribution $\mu$ has order $0$ in $(t,x)$ (Lemma \ref{lem:evolution}),
the double limit is $\langle\mu, g\,\varphi^2\,b_{jl}\rangle$ after the uniform reflection relabel $\check b_{jl}\mapsto b_{jl}$ of the preceding paragraph,
i.e.\ $\mu$ evaluated on the degree-zero symbol $b_{jl} = \xi_j\partial_{\xi_l}a$ carrying the continuous coefficient $\partial_{x_l}b_j$.
Summing over $l$ and $j$ and using $\sum_{l=1}^d \partial_{\xi_l}a\,\partial_{x_l}b_j = \{a, b_j\}$,
this converges to the Poisson bracket $\sum_{j=1}^d \langle \mu, \xi_j \{a, b_j\} \varphi^2 \rangle$. Evaluating the limits of the right-hand sides of both
paired equations yields $\langle \mu, a \partial_t(\varphi^2) \rangle + \sum_{j=1}^d \langle \mu, a b_j \partial_{x_j}(\varphi^2) \rangle$. 

Substituting a non-negative test function $\Psi = \varphi^2$ and equating the limits, we obtain:
\begin{equation*}
    \langle \mu, a \partial_t \Psi \rangle + \langle \mu, \sum_{j=1}^d a b_j \partial_{x_j} \Psi \rangle + \langle \mu, a(\nabla \cdot b)\Psi \rangle - \langle \mu, \sum_{j=1}^d \xi_j \{a, b_j\} \Psi \rangle - \langle \mu, p a c \Psi \rangle = 0.
\end{equation*}

We rewrite the Poisson bracket term using the bicharacteristic frequency flow $\dot{\xi}_k = -\sum_j \xi_j \partial_{x_k} b_j$, observing that
$\sum_{j=1}^d \xi_j \{a, b_j\} = -\dot{\xi} \cdot \nabla_\xi a$. Since $a$ is homogeneous of degree zero, Euler's identity gives
$\xi \cdot \nabla_\xi a = 0$, so $\nabla_\xi a = \nabla_{\Sp^{d-1}} a$ is tangent to the sphere and only the tangential projection $\dot{\xi}^\top$
of the flow contributes: $\dot{\xi} \cdot \nabla_\xi a = \dot{\xi}^\top \cdot \nabla_{\Sp^{d-1}} a$. Performing the integration by parts intrinsically
on the frequency sphere therefore shifts the derivative onto $\mu$ via the surface divergence:
\begin{equation*}
    -\langle \mu, (-\dot{\xi}^\top \cdot \nabla_{\Sp^{d-1}} a)\Psi \rangle = -\langle \operatorname{div}_{\Sp^{d-1}}(\dot{\xi}^\top \mu), a \Psi \rangle.
\end{equation*}

Shifting the spatial and temporal derivatives onto $\mu$ in the sense of distributions yields:
\begin{equation*}
    \langle -\partial_t \mu - b \cdot \nabla_x \mu - \operatorname{div}_{\Sp^{d-1}}(\dot{\xi}^\top \mu) - p c \mu, a\Psi \rangle = 0.
\end{equation*}

This has been established for $\Psi = \varphi^2$ with $\varphi \in C_c^1$; since any real $\Psi \in C_c^\infty(\R\times\R^d)$ admits the polarization
$\Psi = \tfrac14\big[(\chi+\Psi)^2 - (\chi-\Psi)^2\big]$ with $\chi \in C_c^\infty$ equal to $1$ on $\operatorname{supp}\Psi$ (both $\chi\pm\Psi \in C_c^\infty$),
and the left-hand side is linear in $\Psi$, the identity extends to all real $\Psi$ (and to complex $\Psi$ by real and imaginary parts), so it holds in
$\mathcal{D}'(\R\times\R^d\times\Sp^{d-1})$. Expanding the surface divergence gives
$\operatorname{div}_{\Sp^{d-1}}(\dot{\xi}^\top\mu) = \dot{\xi}^\top \cdot \nabla_{\Sp^{d-1}}\mu + (\operatorname{div}_{\Sp^{d-1}}\dot{\xi}^\top)\mu$.
Consequently, along the projected bicharacteristic rays on the sphere the H-distribution propagates according to the material derivative
\begin{equation*}
    \frac{D\mu}{Dt} := \partial_t\mu + b\cdot\nabla_x\mu + \dot{\xi}^\top\cdot\nabla_{\Sp^{d-1}}\mu = -\big(\operatorname{div}_{\Sp^{d-1}}\dot{\xi}^\top + pc\big)\mu.
\end{equation*}
On the full cotangent bundle the underlying Hamiltonian flow generated by $H = b\cdot\xi$ is incompressible
(Liouville's theorem, $\nabla_x\cdot\dot{x} + \nabla_\xi\cdot\dot{\xi} = 0$, whence $\nabla_\xi\cdot\dot{\xi} = -\nabla_x\cdot b$);
the projective reduction to the cosphere bundle replaces this ambient frequency divergence by the intrinsic surface divergence
$\operatorname{div}_{\Sp^{d-1}}\dot{\xi}^\top$, which additionally carries the geometric curvature of the frequency projection.
\end{proof}

\subsection{Physical Interpretation: Hamiltonian Flow and the Vlasov Equation}

The transport equation derived in Theorem \ref{thm:bicharacteristic_first_order} shows that the propagation of the H-distribution $\mu$
is governed by the geometry of classical mechanics. To see this, we define the principal symbol of the spatial transport operator as a phase-space Hamiltonian:
\begin{equation*}
    H(t, x, \xi) = \sum_{j=1}^d b_j(t,x) \xi_j = b(t,x) \cdot \xi.
\end{equation*}

The bicharacteristic rays along which the high-frequency oscillations propagate are precisely the solutions to Hamilton's equations for this system:
\begin{align*}
    \dot{x}_k &= \frac{\partial H}{\partial \xi_k} = b_k(t,x), \\
    \dot{\xi}_k &= -\frac{\partial H}{\partial x_k} = -\sum_{j=1}^d \xi_j \frac{\partial b_j}{\partial x_k}(t,x).
\end{align*}

In classical statistical mechanics, Liouville's theorem states that the phase-space volume of a system evolving under Hamiltonian dynamics is conserved.
This is equivalent to stating that the phase-space velocity field $V = (\dot{x}, \dot{\xi})$ is divergence-free. Computing the phase-space divergence of
our system yields an exact cancellation:
\begin{equation*}
    \nabla_{x,\xi} \cdot V = \sum_{k=1}^d \frac{\partial \dot{x}_k}{\partial x_k} + \sum_{k=1}^d \frac{\partial \dot{\xi}_k}{\partial \xi_k} = \sum_{k=1}^d \frac{\partial b_k}{\partial x_k} - \sum_{k=1}^d \frac{\partial b_k}{\partial x_k} = 0.
\end{equation*}

Because the flow is incompressible ($\nabla_x \cdot \dot{x} + \nabla_\xi \cdot \dot{\xi} = 0$), the transport equation of Theorem
\ref{thm:bicharacteristic_first_order} corresponds to the standard kinetic (Vlasov) form
\begin{equation*}
    \partial_t \mu + \{H, \mu\}_{x,\xi} = (\nabla_x \cdot b - p c(t,x)) \mu,
\end{equation*}
understood formally, as a bracket identity on the algebra of degree-zero symbols rather than a distributional PDE on the punctured cotangent bundle: the
bracket $\{H,\mu\}_{x,\xi}$ and the divergence $\nabla_\xi\cdot\dot\xi$ presuppose a radial extension of $\mu$ off $\Sp^{d-1}$ that we do not construct.

This Vlasov form is the homogeneous (degree-zero) lift of the transport equation of Theorem \ref{thm:bicharacteristic_first_order}
to the full cotangent bundle $\R^d_x \times (\R^d_\xi \setminus \{0\})$; the microlocal object $\mu$ is its projectivisation onto the cosphere
bundle $\R^d_x \times \Sp^{d-1}$, where the ambient frequency divergence $\nabla_\xi \cdot \dot{\xi}$ is replaced by the intrinsic surface divergence
$\operatorname{div}_{\Sp^{d-1}}\dot{\xi}^\top$ of the projected flow. The two coincide, however, only as bracket actions tested against degree-zero
symbols, and not pointwise on the cosphere bundle: for the degree-one frequency field $\dot\xi$ the intrinsic surface divergence
$\operatorname{div}_{\Sp^{d-1}}\dot{\xi}^\top$ and the ambient divergence $\nabla_\xi\cdot\dot\xi = -\nabla_x\cdot b$ differ pointwise by the radial and
mean-curvature contributions discarded in the projection $\dot\xi\mapsto\dot\xi^\top$. The displayed Vlasov equation is therefore to be read as the
degree-zero homogeneous lift (exact once integrated against a homogeneous symbol $a(\xi)$) rather than as an identity of the two damping
coefficients at a fixed point of $\R^d_x\times\Sp^{d-1}$; it is in this tested sense that the Poisson-bracket form above faithfully encodes the
propagation. This formulation is mathematically recognized as a damped Vlasov equation.
In plasma physics and galactic dynamics, the Vlasov equation describes the time evolution of a collisionless particle distribution function in phase space.
Here, the ``particles'' are the microlocal concentrations of the $L^p$ sequence $(u_n)$. 

The left-hand side, $\partial_t \mu + \{H, \mu\}_{x,\xi} = \frac{D\mu}{Dt}$, represents the material derivative, the rate of change observed by riding along
the bicharacteristic rays. The right-hand side represents the continuous attenuation (or amplification) of the distribution, modified by the geometric
expansion of the spatial flow $\nabla_x \cdot b$. This establishes a parallel with the macroscopic PDE. If the macroscopic sequence $u_n$ experiences
exponential damping driven by $e^{-\int c(t,x) dt}$, the H-distribution $\mu$ attenuates at the $p$-th power of that rate, offset by the spatial divergence
of the transport rays.

\section{Application II: $L^p$ Transport Equations for the Quasilinear $p$-Wave System}\label{sec:wave_eq}

In the classical $L^2$ framework, the transport of H-measures for the wave equation (Tartar \cite[Theorem 2.2]{Tartar2017}) relies on a local energy identity
(Tartar's Lemma 2.1). This identity is derived by multiplying the linear wave equation $\partial_t(\rho \partial_t u) - \nabla \cdot (C \nabla u) = 0$ by the
time derivative $\partial_t u$, relying on the algebraic symmetry $\nabla u \cdot \partial_t \nabla u = \frac{1}{2}\partial_t |\nabla u|^2$. 

In the $L^p$ setting ($p \neq 2$), pairing the linear wave equation with the canonical dual $\Phi_p(\partial_t u) = |\partial_t u|^{p-2}\partial_t u$ fails to
produce a closed energy identity. The nonlinear weight $|\partial_t u|^{p-2}$ couples irreversibly with the spatial gradients, reflecting the analytic fact
that the linear wave group does not preserve $L^p$ norms. 

To track the microlocal concentration of waves in Banach spaces, the underlying geometry of the partial differential equation must be adjusted to the
$p$-topology. We formulate the $L^p$ analogue of the wave equation as a coupled, quasilinear $p$-hyperbolic system, establish its pointwise energy identity,
and lift the latter to the microlocal level. The lift closes into a bona fide spatial transport law in the linear core ($p=2$), where the energy is carried by a
microlocal energy-flux (Poynting) vector; for $p \neq 2$ the closure fails for a structural reason that we isolate precisely, and which we record as an open
problem.

\subsection{The Local $L^p$ Energy Identity}

Let $u_n$ be a sequence of solutions. We define the velocity sequence $U_n = \partial_t u_n$ and the spatial gradient sequence $X_n = \nabla u_n$. By the
symmetry of mixed derivatives, we have the linear compatibility constraint $\partial_t X_n - \nabla U_n = 0$. 

To satisfy the generalized chain rule in $L^p$, we restrict the medium to be isotropic. We define the corresponding Isotropic $p$-Wave System on
$\R \times \R^d$ as the coupled pair of equations:
\begin{align}
    \partial_t X_n - \nabla U_n &= 0, \label{eq:pwave1} \\
    \partial_t \big(\rho(x) \Phi_p(U_n)\big) - \nabla \cdot \big(c(x) \Phi_p(X_n)\big) &= 0, \label{eq:pwave2}
\end{align}
where $\rho \in C^{1,\epsilon}(\R^d)$ is a uniformly positive scalar density ($\inf_x \rho > 0$) and $c \in C^{1,\epsilon}(\R^d)$ is a uniformly positive
scalar acoustic wave speed ($\inf_x c > 0$), both with gradients vanishing at spatial infinity in supremum norm
(namely $\nabla \rho, \nabla c \in C_0(\R^d; \R^d)$, uniformly H\"older, as required by Assumption \ref{ass:spatial_reg} for the Second Commutation Lemma),
and $\Phi_p(Z) = |Z|^{p-2}Z$ is the canonical vector Nemyckij operator \cite{AntonicErcegMisur2021} mapping $L^p \to L^{p'}$.

The well-posedness of this degenerate quasilinear system for $p \neq 2$ lies outside our scope: Lemma \ref{lemma:energy} and its microlocal lift are stated
conditionally on a given sequence of solutions, and the microlocal transport theorem below concerns only the linear core $p = 2$, where the system is
the standard first-order symmetric-hyperbolic form of the wave equation.

\begin{lemma}[$L^p$ Energy Identity]\label{lemma:energy}
For $p \ge 2$, any smooth solution to the isotropic $p$-Wave system \eqref{eq:pwave1}-\eqref{eq:pwave2} satisfies the exact local energy identity:
\begin{equation*}
    \partial_t \left( \frac{p-1}{p} \rho(x) |U_n|^p + \frac{1}{p} c(x) |X_n|^p \right) - \nabla \cdot \left( U_n c(x) \Phi_p(X_n) \right) = 0.
\end{equation*}
For $1 < p < 2$ the derivative $\Phi_p'$ is singular at the origin, so the direct differentiation of the dual fields $\Phi_p(U_n), \Phi_p(X_n)$
performed below is delicate near the nodal sets $\{U_n = 0\}, \{X_n = 0\}$ (cf.\ Remark \ref{rem:chainrange}). The stated
identity, however, involves only $|U_n|^p$ and $|X_n|^p$, which are $C^1$ for every $p > 1$, and that its temporal term may alternatively be obtained through the dual
convex potential $\tfrac1{p'}|\Phi_p(U_n)|^{p'} = \tfrac1{p'}|U_n|^p$ without differentiating $\Phi_p$; for a smooth solution with transversal zeros every term
is $L^1_{loc}$ when $p > 1$ (as $p - 2 > -1$), so the identity plausibly persists into $1 < p < 2$ under a suitable nodal-set hypothesis. We nonetheless state
it only for $p \ge 2$, where the pointwise chain rule for $\Phi_p$ is unconditionally available, and defer the careful low-regularity treatment to Section
\ref{sec:future}. The microlocal result below is in any case confined to the linear core $p = 2$.
\end{lemma}

\begin{proof}
We pair equation \eqref{eq:pwave2} with the velocity sequence $U_n \in L_{loc}^p$. For the temporal term, we apply the chain rule to the Nemyckij operator:
\begin{align*}
    U_n \cdot \partial_t \big(\rho \Phi_p(U_n)\big) &= \rho U_n \cdot \partial_t(|U_n|^{p-2}U_n) \\
    &= \rho U_n \cdot (p-1)|U_n|^{p-2}\partial_t U_n \\
    &= \frac{p-1}{p} \partial_t \big( \rho |U_n|^p \big).
\end{align*}

For the spatial term, we use the product rule:
\begin{equation*}
    - U_n \cdot \left( \nabla \cdot (c \Phi_p(X_n)) \right) = - \nabla \cdot (U_n c \Phi_p(X_n)) + \nabla U_n \cdot (c \Phi_p(X_n)).
\end{equation*}

We substitute the compatibility condition $\nabla U_n = \partial_t X_n$ from equation \eqref{eq:pwave1} into the remainder:
\begin{equation*}
    \partial_t X_n \cdot (c(x) \Phi_p(X_n)) = c(x) \left( \partial_t X_n \cdot |X_n|^{p-2}X_n \right) = \frac{1}{p} \partial_t \big( c(x) |X_n|^p \big).
\end{equation*}

Because $c(x)$ is a scalar, it commutes with the dot product. Summing the temporal and spatial terms yields the stated derivative identity.
Only $\Phi_p \in C^1$ enters here (through the chain rule for $\partial_t\Phi_p(U_n)$ and $\partial_t\Phi_p(X_n)$), so no second-derivative
regularity of $\Phi_p$ is required, and the identity holds for every $p \ge 2$ (in particular for $2 \le p < 3$).
\end{proof}

\subsection{Microlocal Energy Transport: the Linear Core and the Nonlinear Obstruction}

The local energy identity of Lemma \ref{lemma:energy} is the pointwise conservation law $\partial_t e_n - \nabla\cdot\mathbf{F}_n = 0$, with energy density
$e_n = \frac{p-1}{p}\rho|U_n|^p + \frac{1}{p}c|X_n|^p$ and Poynting flux $\mathbf{F}_n = U_n\, c\, \Phi_p(X_n)$. We lift this identity to the microlocal level.
Let $\boldsymbol{\mu}$ be the block H-distribution generated by the canonical pair $(W_n, Z_n)$, with $W_n = (U_n, X_n)$ and $Z_n = (\Phi_p(U_n), \Phi_p(X_n))$,
in the sense of Lemma \ref{lem:canonical} applied componentwise; write $\mu_U$, $\mu_X$, $\mu_{U,X_k}$, $\mu_{X_k,U}$ for the H-distributions of the pairs
$(U_n, \Phi_p(U_n))$, $(X_{k,n}, \Phi_p(X_n)_l)$, $(U_n, \Phi_p(X_n)_k)$, $(X_{k,n}, \Phi_p(U_n))$. Set
\begin{equation*}
    \nu := \frac{p-1}{p}\rho\,\mu_U + \frac{1}{p}c\,\operatorname{Tr}(\mu_X), \qquad \boldsymbol{\pi} := (\pi_1, \dots, \pi_d), \quad \pi_k := \tfrac12\big(\mu_{U,X_k} + \mu_{X_k,U}\big),
\end{equation*}
for the microlocal energy density and the (symmetrised) microlocal Poynting (momentum) vector; the diagonal of $\nu$ is the microlocal lift of the energy,
$\langle\nu, \varphi^2\rangle = \lim_n\int e_n\varphi^2$. The symmetrisation of $\boldsymbol{\pi}$ is forced by the bilinear nature of the H-distribution
pairing: transferring $\mathcal{A}_a$ from one slot to the other is the transpose $\mathcal{A}_a^\top = \mathcal{A}_{\check a}$ with $\check a(\xi) = a(-\xi)$,
so $\mu_{X_k,U}[a] = \mu_{U,X_k}[\check a]$ and only the even-in-$\xi$ part of the flux is intrinsic (see the proof of Theorem \ref{thm:bicharacteristic_pwave});
as $\nu$ is a diagonal object, hence even in $\xi$, this symmetrisation discards no information.

For the linear wave equation ($p = 2$) this lift closes and yields a genuine spatial transport law:
contrary to what a refraction-only reading would suggest, the microlocal energy is carried through space by the Poynting flux and is not stationary in time.

\begin{theorem}[Microlocal energy transport, linear wave case; recovery of the classical $L^2$ law]\label{thm:bicharacteristic_pwave}
Let $p = 2$, and let $W_n = (U_n, X_n) \rightharpoonup 0$ in $L^2_{loc}(\R\times\R^d; \R^{1+d})$ be solutions of the isotropic wave system
\eqref{eq:pwave1}--\eqref{eq:pwave2} with constant coefficients $\rho, c > 0$, normalised to $\rho = c = 1$, each of class $W^{1,2}_{loc}(\R\times\R^d)$
(no uniform bound on the gradients being imposed; cf.\ Remark \ref{rem:novacuity}). Let $\boldsymbol{\mu}$ be the matrix H-measure of $W_n$, furnished
componentwise by Lemma \ref{lem:evolution}, and set $\nu = \frac{1}{2}\mu_U + \frac{1}{2}\operatorname{Tr}(\mu_X)$ and
$\pi_k = \tfrac12(\mu_{U,X_k}+\mu_{X_k,U})$. Then the microlocal energy is transported in space by the (symmetrised) Poynting flux:
\begin{equation*}
    \partial_t \nu - \operatorname{div}_x \boldsymbol{\pi} = 0
\end{equation*}
in the sense of distributions on $\R\times\R^d\times\Sp^{d-1}$.
\end{theorem}

\begin{proof}
Fix a real symbol $a\in C^\infty(\Sp^{d-1})$ and a test function $\Psi = \varphi^2$ with $\varphi\in C_c^1(\R\times\R^d)$;
pass without relabelling to the subsequence along which the H-measure $\boldsymbol{\mu}$ of Lemma \ref{lem:evolution} exists.
(In the bilinear H-measure pairing the relevant adjoint of $\mathcal{A}_a$ is the transpose
$\mathcal{A}_a^\top = \mathcal{A}_{\check a}$, $\check a(\xi)=a(-\xi)$, not a Hilbert adjoint; the diagonal objects $\mu_U, \mu_X$ are
even in $\xi$ and hence insensitive to the distinction, while the cross terms are handled by the symmetrisation in $\boldsymbol{\pi}$ below.)
Its hypotheses hold for the wave system: (i) at $p = 2$, $\rho = c = 1$ the energy $e_n = \tfrac12|W_n|^2$ obeys the conservation law
$\partial_t e_n = \nabla\cdot(U_n X_n)$ of Lemma \ref{lemma:energy}, so over the backward light cone (unit speed) $\int_B e_n(t) \le \int_{B'} e_n(s)$
for $s < t$ in $J$, $B \subset K$ and $B'$ the enlarged cone base; averaging $s$ over $J'$ against
$\int_{J'}\!\int_{B'} e_n = \tfrac12\norm{W_n}_{L^2(J'\times B')}^2$ selects one controlled slice and yields
$\sup_{t\in J}\int_K |W_n(t,\cdot)|^2 \le C\,\norm{W_n}_{L^2(J'\times K')}^2$ over a slightly enlarged cone, uniformly in $n$; and (ii)
follows from $\partial_t U_n = \nabla\cdot X_n$ and $\partial_t X_{k,n} = \partial_{x_k} U_n$, since
$\langle W_n(t,\cdot) - W_n(s,\cdot), \theta\rangle = \int_s^t \langle W_n, \mathcal{L}^*\theta\rangle\,\rd\tau$ with $\mathcal{L}^*$
a fixed first-order operator, bounded by $C(\theta)|t-s|$ via (i). Since $\Phi_2 = \operatorname{id}$, the dual sequence coincides with the primal one,
$Z_n = W_n$, and the system reads
\begin{equation*}
    \partial_t X_{k,n} = \partial_{x_k} U_n, \qquad \partial_t U_n = \nabla\cdot X_n.
\end{equation*}
\emph{Spatial localization.} As $U_n, X_n$ are only locally square-integrable, every $\mathcal{A}_a$ below is understood to act on the fields localized
by a fixed $\chi\in C_c^\infty(\R^d)$ with $\chi\equiv 1$ on a neighbourhood of the $x$-projection of $\operatorname{supp}\varphi$, so that
$\mathcal{A}_a(\chi U_n), \mathcal{A}_a(\chi X_{k,n})\in L^2(\R^d)$ are well-defined. Every discrepancy so introduced (the long-range part
$\mathcal{A}_a((1-\chi)\,\cdot\,)$ and the cutoff commutator
$\mathcal{A}_a((\partial_{x_k}\chi)U_n) = \partial_{x_k}\mathcal{A}_a(\chi U_n) - \mathcal{A}_a(\chi\partial_{x_k}U_n)$) is, on
$\operatorname{supp}\Psi\cup\operatorname{supp}\nabla\Psi$, an integral operator with kernel $\chi_0(x)\,k_a(x-y)\,(1-\chi(y))$
(resp.\ $\chi_0(x)\,k_a(x-y)\,\partial_{x_k}\chi(y)$) whose $x$- and $y$-supports are separated, hence bounded and rapidly decaying;
tested against the fiberwise weakly-null slices $W_n(t,\cdot)\rightharpoonup 0$ (Lemma \ref{lem:evolution}, Step 1) these Hille--Tamarkin
operators contribute $0$ to every limit by dominated convergence. We therefore suppress $\chi$ and compute as though $\mathcal{A}_a$ acted on global $L^2$ fields.

By the definition of the H-measure and integration by parts in $t$,
\begin{equation*}
    \langle\partial_t\nu, a\Psi\rangle = -\langle\nu, a\partial_t\Psi\rangle = \frac12\lim_n\int \Psi\,\partial_t\Big[U_n\mathcal{A}_a U_n + \sum_{k=1}^d X_{k,n}\mathcal{A}_a X_{k,n}\Big].
\end{equation*}
Using the two equations and $[\mathcal{A}_a, \partial_{x_k}] = 0$, the integrand telescopes into a spatial divergence:
\begin{align*}
    \partial_t\Big[U_n\mathcal{A}_a U_n + \sum_k X_{k,n}\mathcal{A}_a X_{k,n}\Big]
    &= \underbrace{(\nabla\cdot X_n)\,\mathcal{A}_a U_n + \sum_k X_{k,n}\,\mathcal{A}_a\partial_{x_k}U_n}_{=\,\operatorname{div}_x(\mathcal{A}_a U_n\, X_n)} \\
    &\quad + \underbrace{U_n\,\mathcal{A}_a(\nabla\cdot X_n) + \sum_k (\partial_{x_k}U_n)\,\mathcal{A}_a X_{k,n}}_{=\,\operatorname{div}_x(U_n\, \mathcal{A}_a X_n)}.
\end{align*}
Substituting and integrating by parts in $x$,
\begin{equation*}
    \langle\partial_t\nu, a\Psi\rangle = -\frac12\lim_n\int \nabla\Psi\cdot\big(\mathcal{A}_a U_n\, X_n + U_n\, \mathcal{A}_a X_n\big) = -\frac12\sum_{k=1}^d\big\langle\mu_{X_k,U} + \mu_{U,X_k},\, a\,\partial_{x_k}\Psi\big\rangle,
\end{equation*}
the first half ($\mathcal{A}_a$ acting on $U_n$) contributing $\mu_{X_k,U}[a]$ and the second ($\mathcal{A}_a$ on $X_{k,n}$) contributing $\mu_{U,X_k}[a]$.
We do not identify these two cross H-measures: the pairing is bilinear, so transferring $\mathcal{A}_a$ between slots is the transpose
$\mathcal{A}_a^\top = \mathcal{A}_{\check a}$ ($\check a(\xi)=a(-\xi)$), yielding only $\mu_{X_k,U}[a] = \mu_{U,X_k}[\check a]$;
the two coincide solely on the even part $\tfrac12(a+\check a)$ of the symbol. Their average is exactly the symmetrised flux
$\pi_k = \tfrac12(\mu_{U,X_k}+\mu_{X_k,U})$, so the right-hand side equals $\langle\operatorname{div}_x\boldsymbol{\pi}, a\Psi\rangle$,
whence $\partial_t\nu = \operatorname{div}_x\boldsymbol{\pi}$ (the identity, established for $\Psi = \varphi^2$, $\varphi\in C_c^1$, extends to all real
$\Psi\in C_c^\infty$ by the polarization $\Psi = \tfrac14[(\chi+\Psi)^2-(\chi-\Psi)^2]$, $\chi\in C_c^\infty$ equal to $1$ on $\operatorname{supp}\Psi$, and
linearity). No information is lost by the symmetrisation: $\nu$ is a diagonal object,
hence even in $\xi$, so $\langle\partial_t\nu, a\Psi\rangle$ depends only on the even part of $a$, and the odd part of $\boldsymbol{\pi}$
is invisible to the identity.
\end{proof}

\begin{remark}[Relation to the classical $L^2$ theory]\label{rem:wave-classical}
Theorem \ref{thm:bicharacteristic_pwave} is not new. At $p = 2$ with constant coefficients the block $\boldsymbol{\mu}$ is the classical Tartar
H-measure of the constant-coefficient wave equation, and $\partial_t\nu = \operatorname{div}_x\boldsymbol{\pi}$ is the $x$-cosphere ($\tau$-marginal)
projection of its transport along the null bicharacteristics of $\tau^2 = |\xi|^2$ \cite{Tartar1990, Tartar2017}. The proof invokes none of the
present paper's new machinery: the Second Commutation Lemma (Theorem \ref{thm:second_comm}) is superfluous here because $\mathcal{A}_a$ commutes with
the constant-coefficient evolution, and the $L^p$ H-distribution framework collapses to an $L^2$ H-measure since $\Phi_2 = \operatorname{id}$.
We include the theorem only to fix the notation $\nu, \boldsymbol{\pi}$ and to exhibit the closed transport law in the one regime where it holds,
thereby setting up, by contrast, the structural failure at $p \neq 2$ isolated in Remark \ref{rem:nonlinear-obstruction}. The Second Commutation Lemma
enters the wave problem through the variable-coefficient refraction computation, which we carry out in Theorem \ref{thm:wave-variable}.
\end{remark}

\begin{theorem}[Microlocal energy transport with refraction: the variable-coefficient linear wave case]\label{thm:wave-variable}
Let $p = 2$, and let $\rho, c \in C^{1,\epsilon}(\R^d)$ be uniformly elliptic and bounded, $0 < \inf_x\rho \le \rho \le \sup_x\rho < \infty$ and likewise
for $c$, with $\nabla\rho, \nabla c \in C_0(\R^d;\R^d)$ uniformly H\"older, so that $\rho, c$ and $\rho^{-1}$ each satisfy Assumptions \ref{ass:spatial_reg}
and \ref{ass:spatial_bound}. Let $W_n = (U_n, X_n) \rightharpoonup 0$ in $L^2_{loc}(\R\times\R^d;\R^{1+d})$ be solutions of the isotropic wave system
\eqref{eq:pwave1}--\eqref{eq:pwave2}, each of class $W^{1,2}_{loc}(\R\times\R^d)$ (no uniform bound on the gradients being imposed;
cf.\ Remark \ref{rem:novacuity}). Let $\boldsymbol{\mu}$ be the block H-measure of $W_n$ furnished componentwise by Lemma \ref{lem:evolution}, and set
$\nu = \tfrac12\rho\,\mu_U + \tfrac12 c\operatorname{Tr}(\mu_X)$ and $\pi_k = \tfrac12(\mu_{U,X_k} + \mu_{X_k,U})$. Then the microlocal energy obeys the
spatial transport law with refraction
\begin{equation}\label{eq:wave-variable}
    \partial_t\nu - \operatorname{div}_x(c\,\boldsymbol{\pi}) = \mathcal{R}
\end{equation}
in $\mathcal{D}'(\R\times\R^d\times\Sp^{d-1})$, where the refraction functional $\mathcal{R}$ acts on real symbols $a\in C^\infty(\Sp^{d-1})$ and test functions
$\Psi\in C_c^\infty(\R\times\R^d)$ by
\begin{equation}\label{eq:refraction}
    \langle\mathcal{R}, a\,\Psi\rangle = \frac12\sum_{k,l=1}^d\Big\langle \pi_k,\ \rho\,\partial_{x_l}\!\big(c/\rho\big)\,\xi_k\,\partial_{\xi_l}a\ \Psi\Big\rangle.
\end{equation}
The symbols $\xi_k\partial_{\xi_l}a$ are homogeneous of degree zero, so \eqref{eq:refraction} is a genuine pairing of the (symmetrised) cross-component
$\pi_k$ of $\boldsymbol{\mu}$ against a degree-zero symbol, manifestly even in $\xi$ as the left-hand side of \eqref{eq:wave-variable} demands; $\mathcal{R}$
is first order in the frequency variable $\xi$, is supported on $\nabla\rho, \nabla c$, and vanishes identically when $\rho, c$ are constant, whereupon
\eqref{eq:wave-variable} reduces to $\partial_t\nu=\operatorname{div}_x(c\,\boldsymbol{\pi})$, and, in the normalisation $\rho=c=1$, to Theorem
\ref{thm:bicharacteristic_pwave}.
\end{theorem}

\begin{proof}
Throughout, $a\in C^\infty(\Sp^{d-1})$ is real and $\Psi\in C_c^\infty(\R\times\R^d)$; $\mathcal{A}_a$ acts in $x$ alone.
Abbreviate $b_{kl}(\xi) := \xi_k\,\partial_{\xi_l}a(\xi)$, homogeneous of degree zero and, since $a\in C^\infty$, smooth on $\Sp^{d-1}$; thus each
$\mathcal{A}_{b_{kl}}$ satisfies Assumption \ref{ass:symbol}, and together with $M_c, M_\rho, M_{\rho^{-1}}$ (which satisfy Assumptions
\ref{ass:spatial_reg}--\ref{ass:spatial_bound}) falls under the Second Commutation Lemma (Theorem \ref{thm:second_comm}), the inhomogeneous bound
(Theorem \ref{thm:inhomogeneous}), and the building blocks (Lemma \ref{lem:building}). Since $\rho$ and $c$ are independent of $t$, every operator
built from them below ($[\mathcal{A}_a,M_c]$, $[\mathcal{A}_a,M_{\rho^{-1}}]$, the principal parts $S_k(\cdot)$ and the compact remainders
$R_k(\cdot)$) is itself independent of $t$, so their $L^2$-bounds and compactness are uniform in $t$ trivially (the parametric uniformity of Remark
\ref{rem:parameter} is not needed here).

\emph{Existence, localization, reduction.} Since $\Phi_2 = \operatorname{id}$, the dual sequence coincides with the primal, $Z_n = W_n$, and the system reads
\begin{equation}\label{eq:pwave-p2}
    \partial_t X_{k,n} = \partial_{x_k}U_n, \qquad \rho\,\partial_t U_n = \nabla\cdot(cX_n) = \sum_{m}\partial_{x_m}(cX_{m,n}).
\end{equation}
The hypotheses of Lemma \ref{lem:evolution} hold as in Theorem \ref{thm:bicharacteristic_pwave}: (i) the local energy identity of Lemma \ref{lemma:energy} at
$p = 2$, with $e_n = \tfrac12\rho U_n^2 + \tfrac12 c|X_n|^2 \ge \tfrac12\min(\inf\rho, \inf c)\,|W_n|^2$ and flux $U_n cX_n$: over the backward cone of speed
$\sqrt{\sup(c/\rho)}$ the coercive energy obeys $\int_B e_n(t) \le \int_{B'} e_n(s)$ ($s < t$ in $J$, $B'$ the enlarged base), and averaging $s$ over $J'$
against $\int_{J'}\!\int_{B'} e_n \le C\norm{W_n}_{L^2(J'\times K')}^2$ bounds $\sup_{t\in J}\norm{W_n(t,\cdot)}_{L^2(K)}$ uniformly in $n$; (ii) equicontinuity of the slices follows from \eqref{eq:pwave-p2}, whose right-hand sides
are fixed first-order operators with bounded coefficients. Hence $\boldsymbol{\mu}$ exists (along a subsequence, not relabelled), each component of order $0$ in
$(t,x)$. As in Theorem \ref{thm:bicharacteristic_pwave} we localize every $\mathcal{A}_a$ and $\mathcal{A}_{b_{kl}}$ by a fixed $\chi\in C_c^\infty(\R^d)$ equal
to $1$ on a neighbourhood of the $x$-projection of $\operatorname{supp}\Psi$; the induced Hille--Tamarkin discrepancies have kernels with separated supports and
contribute $0$ to every limit by fiberwise weak nullity and dominated convergence. We suppress $\chi$ below.

We use repeatedly the following \emph{fiberwise compactness principle}, established in the proof of Theorem \ref{thm:bicharacteristic_first_order}
(paragraph ``The remainder pairings vanish'') and again in Theorem \ref{thm:bicharacteristic_pwave}: if $K$ is a fixed operator, compact on $L^2(\R^d)$
(all the operators $K$ to which we apply this are $t$-independent, being built from the $t$-independent $\rho, c$), then for slices
$Y_n(t,\cdot)\rightharpoonup 0$ in $L^2$ and $Z_n(t,\cdot)$ bounded in $L^2$ uniformly on $\operatorname{supp}\Psi$,
\begin{equation}\label{eq:fiber}
    \lim_n\int \Psi\, Z_n\, K(t)Y_n\,\rd x\,\rd t = 0,
\end{equation}
since $K Y_n(t,\cdot)\to 0$ strongly in $L^2$ for a.e.\ $t$ (compactness together with the fiberwise weak nullity of Lemma \ref{lem:evolution}, Step~1)
and the space integral is dominated by $C\operatorname*{ess\,sup}_t\norm{Y_n(t)}_{L^2}\operatorname*{ess\,sup}_t\norm{Z_n(t)}_{L^2}\mathbf 1_J\in L^1(\rd t)$.
Here and below the fiberwise weak nullity, the uniform slice bounds, and the everywhere-defined weakly continuous slice representatives are those of a
single subsequence, extracted once by the finite diagonal argument of Lemma \ref{lem:evolution} so as to serve the block $\boldsymbol{\mu}$ and all its
cross-components $\mu_U,\mu_{U,X_k},\mu_{X_k,U},\mu_{X_k,X_l}$ simultaneously; every appearance of $Y_n\in\{U_n,X_{k,n},cX_{m,n}\}$ and $Z_n=U_n$ in
\eqref{eq:fiber} refers to these same representatives.

By the definition of $\nu$, the $t$-independence of $\rho, c$, and integration by parts in $t$,
\begin{equation}\label{eq:dtnu}
    \langle\partial_t\nu, a\Psi\rangle = -\langle\nu, a\,\partial_t\Psi\rangle = \tfrac12\lim_n\int \Psi\,\partial_t D_n, \qquad D_n := \rho\,U_n\mathcal{A}_a U_n + c\sum_{k}X_{k,n}\mathcal{A}_a X_{k,n},
\end{equation}
the differentiation $\partial_t D_n$ being licit for each $n$ since $U_n, X_n\in W^{1,2}_{loc}$ and $\mathcal{A}_a$ commutes with $\partial_t$.

\emph{Step 1: exact telescoping.} Using \eqref{eq:pwave-p2}, $[\mathcal{A}_a,\partial_{x_k}] = 0$, and the exact identities
$\rho\,\mathcal{A}_a g = \mathcal{A}_a(\rho g) - [\mathcal{A}_a,M_\rho]g$ and
$\mathcal{A}_a(cX_{k,n}) = c\,\mathcal{A}_a X_{k,n} + [\mathcal{A}_a,M_c]X_{k,n}$, we obtain the pointwise-a.e.\ identity
\begin{align}
    \partial_t D_n
    &= \underbrace{[\nabla\cdot(cX_n)]\,\mathcal{A}_a U_n + \sum_k cX_{k,n}\,\partial_{x_k}\mathcal{A}_a U_n}_{=\,\operatorname{div}_x(c\,\mathcal{A}_a U_n\,X_n)} \notag\\
    &\quad + \underbrace{U_n\,\mathcal{A}_a(\nabla\cdot(cX_n)) + \sum_k c(\partial_{x_k}U_n)\,\mathcal{A}_a X_{k,n}}_{=\,\operatorname{div}_x(U_n\,\mathcal{A}_a(cX_n))\ -\ \sum_k(\partial_{x_k}U_n)[\mathcal{A}_a,M_c]X_{k,n}} \notag\\
    &\quad - U_n\,[\mathcal{A}_a,M_\rho]\,\partial_t U_n. \label{eq:telescope}
\end{align}
The first bracket is an exact Leibniz divergence. For the second, $\mathcal{A}_a(\nabla\cdot(cX_n)) = \sum_k\partial_{x_k}\mathcal{A}_a(cX_{k,n})$;
writing $\sum_k U_n\partial_{x_k}\mathcal{A}_a(cX_{k,n}) = \operatorname{div}_x(U_n\mathcal{A}_a(cX_n)) - \sum_k(\partial_{x_k}U_n)\mathcal{A}_a(cX_{k,n})$
and inserting $\mathcal{A}_a(cX_{k,n}) = c\mathcal{A}_a X_{k,n} + [\mathcal{A}_a,M_c]X_{k,n}$ gives the stated underbrace. In the last term,
$\partial_t U_n = \rho^{-1}\nabla\cdot(cX_n)$ by \eqref{eq:pwave-p2}.

Substituting \eqref{eq:telescope} into \eqref{eq:dtnu} and integrating the two divergences by parts against $\Psi$,
\begin{equation}\label{eq:assembled}
    \langle\partial_t\nu, a\Psi\rangle = -\tfrac12\sum_k\lim_n\int (\partial_{x_k}\Psi)\big[c\,\mathcal{A}_a U_n\,X_{k,n} + U_n\,\mathcal{A}_a(cX_{k,n})\big] + \mathcal{R}^{(2)} + \mathcal{R}^{(3)},
\end{equation}
where
\begin{equation*}
    \mathcal{R}^{(2)} := -\tfrac12\sum_k\lim_n\int \Psi\,(\partial_{x_k}U_n)[\mathcal{A}_a,M_c]X_{k,n}, \qquad
    \mathcal{R}^{(3)} := -\tfrac12\lim_n\int \Psi\,U_n[\mathcal{A}_a,M_\rho]\partial_t U_n.
\end{equation*}

\emph{Step 2: the flux limit.} In the bracket of \eqref{eq:assembled} split $\mathcal{A}_a(cX_{k,n}) = c\mathcal{A}_a X_{k,n} + [\mathcal{A}_a,M_c]X_{k,n}$.
The two ``clean'' pieces converge by the definition of $\boldsymbol{\mu}$ (the continuous coefficient $c$ being admissible since $\mu_{U,X_k},\mu_{X_k,U}$
have order $0$ in $(t,x)$):
\begin{equation*}
    \lim_n\int(\partial_{x_k}\Psi)\,c\,\mathcal{A}_a U_n\,X_{k,n} = \langle\mu_{X_k,U}, c\,a\,\partial_{x_k}\Psi\rangle, \quad
    \lim_n\int(\partial_{x_k}\Psi)\,U_n\,c\,\mathcal{A}_a X_{k,n} = \langle\mu_{U,X_k}, c\,a\,\partial_{x_k}\Psi\rangle,
\end{equation*}
($\mathcal{A}_a$ on the $U$-slot yielding $\mu_{X_k,U}$, on the $X$-slot yielding $\mu_{U,X_k}$). Their half-sum is $\langle c\,\pi_k, a\,\partial_{x_k}\Psi\rangle$, and $-\sum_k\langle c\pi_k, a\partial_{x_k}\Psi\rangle = \langle\operatorname{div}_x(c\boldsymbol{\pi}), a\Psi\rangle$ (as $a$ is $\xi$-only). The leftover commutator piece is
\begin{equation*}
    \mathcal{R}^{(1)} := -\tfrac12\sum_k\lim_n\int(\partial_{x_k}\Psi)\,U_n[\mathcal{A}_a,M_c]X_{k,n},
\end{equation*}
so that $\langle\partial_t\nu - \operatorname{div}_x(c\boldsymbol{\pi}), a\Psi\rangle = \mathcal{R}^{(1)}+\mathcal{R}^{(2)}+\mathcal{R}^{(3)}$.

\emph{Step 3: the $c$-refraction $\mathcal{R}^{(1)}+\mathcal{R}^{(2)}$.} Since $U_n\in W^{1,2}_{loc}$ and $[\mathcal{A}_a,M_c]X_{k,n}\in W^{1,2}_{loc}$
(Theorem \ref{thm:inhomogeneous}, $c\in C^{1,\epsilon}\cap L^\infty$), integration by parts in $x_k$ gives
\begin{equation*}
    \int\Psi(\partial_{x_k}U_n)[\mathcal{A}_a,M_c]X_{k,n} = -\int U_n(\partial_{x_k}\Psi)[\mathcal{A}_a,M_c]X_{k,n} - \int U_n\Psi\,\partial_{x_k}\!\big([\mathcal{A}_a,M_c]X_{k,n}\big),
\end{equation*}
whose first summand cancels $\mathcal{R}^{(1)}$ exactly, leaving
\begin{equation*}
    \mathcal{R}^{(1)}+\mathcal{R}^{(2)} = \tfrac12\sum_k\lim_n\int U_n\,\Psi\,\partial_{x_k}\!\big([\mathcal{A}_a,M_c]X_{k,n}\big).
\end{equation*}
By the Second Commutation Lemma (Theorem \ref{thm:second_comm}(II)), $\partial_{x_k}[\mathcal{A}_a,M_c] = S_k(c) + R_k(c)$ with $R_k(c)$ compact on $L^2$ and
$S_k(c) = \operatorname{Op}\big(\xi_k\sum_l\partial_{\xi_l}a\,\partial_{x_l}c\big) = \sum_l M_{\partial_{x_l}c}\,\mathcal{A}_{b_{kl}}$.
The remainder $R_k(c)X_{k,n}$ contributes $0$ by \eqref{eq:fiber}; hence, using $\partial_{x_l}c\in C_0\subset C$,
\begin{equation}\label{eq:R12}
    \mathcal{R}^{(1)}+\mathcal{R}^{(2)} = \tfrac12\sum_{k,l}\lim_n\int (\partial_{x_l}c)\,\Psi\,U_n\,\mathcal{A}_{b_{kl}}X_{k,n} = \tfrac12\sum_{k,l}\big\langle\mu_{U,X_k},\ (\partial_{x_l}c)\,\xi_k\partial_{\xi_l}a\,\Psi\big\rangle.
\end{equation}

\emph{Step 4: the $\rho$-refraction $\mathcal{R}^{(3)}$.} From $M_\rho M_{\rho^{-1}} = I$ one has the exact operator identity
$[\mathcal{A}_a,M_\rho]M_{\rho^{-1}} = -M_\rho[\mathcal{A}_a,M_{\rho^{-1}}]$; with $\partial_t U_n = \rho^{-1}\nabla\cdot(cX_n)$
this gives $U_n[\mathcal{A}_a,M_\rho]\partial_t U_n = -U_n\,\rho\,[\mathcal{A}_a,M_{\rho^{-1}}]\nabla\cdot(cX_n)$, so
\begin{equation*}
    \mathcal{R}^{(3)} = \tfrac12\sum_m\lim_n\int (\rho\Psi)\,U_n\,[\mathcal{A}_a,M_{\rho^{-1}}]\,\partial_{x_m}(cX_{m,n}).
\end{equation*}
Using $[\mathcal{A}_a,M_{\rho^{-1}}]\partial_{x_m} = \partial_{x_m}[\mathcal{A}_a,M_{\rho^{-1}}] - [\mathcal{A}_a,M_{\partial_{x_m}\rho^{-1}}]$,
\begin{equation*}
    [\mathcal{A}_a,M_{\rho^{-1}}]\partial_{x_m}(cX_{m,n}) = \partial_{x_m}\!\big([\mathcal{A}_a,M_{\rho^{-1}}](cX_{m,n})\big) - [\mathcal{A}_a,M_{\partial_{x_m}\rho^{-1}}](cX_{m,n}).
\end{equation*}
Since $\partial_{x_m}\rho^{-1} = -\rho^{-2}\partial_{x_m}\rho\in C^\epsilon\cap C_0$, the commutator $[\mathcal{A}_a,M_{\partial_{x_m}\rho^{-1}}]$ is
compact on $L^2$ (Lemma \ref{lem:building}(2)), so its term vanishes by \eqref{eq:fiber}. For the first term, Theorem \ref{thm:second_comm}(II) gives
$\partial_{x_m}[\mathcal{A}_a,M_{\rho^{-1}}] = S_m(\rho^{-1}) + R_m(\rho^{-1})$, $R_m(\rho^{-1})$ compact (vanishing by \eqref{eq:fiber}) and
$S_m(\rho^{-1}) = \sum_l M_{\partial_{x_l}\rho^{-1}}\mathcal{A}_{b_{ml}}$. In
\begin{equation*}
    S_m(\rho^{-1})(cX_{m,n}) = \sum_l(\partial_{x_l}\rho^{-1})\big(c\,\mathcal{A}_{b_{ml}}X_{m,n} + [\mathcal{A}_{b_{ml}},M_c]X_{m,n}\big),
\end{equation*}
the operator $M_{\partial_{x_l}\rho^{-1}}[\mathcal{A}_{b_{ml}},M_c]$ is compact on $L^2$: indeed $[\mathcal{A}_{b_{ml}},M_c]\in\mathcal{L}(L^2,W^{1,2})$
(Theorem \ref{thm:inhomogeneous}) and $\partial_{x_l}\rho^{-1}\in C_0$, so Lemma \ref{lem:building}(1) (with $\sigma = 1$) applies; this term too vanishes
by \eqref{eq:fiber}. The surviving contribution, using $\rho\,\partial_{x_l}\rho^{-1} = -\rho^{-1}\partial_{x_l}\rho$, is
\begin{equation*}
    \mathcal{R}^{(3)} = \tfrac12\sum_{k,l}\lim_n\int \big(\rho c\,\partial_{x_l}\rho^{-1}\,\Psi\big)\,U_n\,\mathcal{A}_{b_{kl}}X_{k,n} = -\tfrac12\sum_{k,l}\big\langle\mu_{U,X_k},\ c\rho^{-1}(\partial_{x_l}\rho)\,\xi_k\partial_{\xi_l}a\,\Psi\big\rangle.
\end{equation*}

\emph{Step 5: assembly.} Adding \eqref{eq:R12} and $\mathcal{R}^{(3)}$ and using $\partial_{x_l}c - c\rho^{-1}\partial_{x_l}\rho = \rho\,\partial_{x_l}(c/\rho)$,
\begin{equation*}
    \mathcal{R}^{(1)}+\mathcal{R}^{(2)}+\mathcal{R}^{(3)} = \tfrac12\sum_{k,l}\big\langle\mu_{U,X_k},\ \rho\,\partial_{x_l}(c/\rho)\,\xi_k\partial_{\xi_l}a\,\Psi\big\rangle.
\end{equation*}
The left-hand side of \eqref{eq:wave-variable} is even in $\xi$ ($\nu$ is diagonal and $\boldsymbol{\pi}$ symmetrised), so this functional of $a$ depends only
on its even part $a_e = \tfrac12(a+\check a)$, $\check a(\xi) := a(-\xi)$. Writing $\check m(\xi):=m(-\xi)$ for the frequency reflection, the transpose relation
$\mu_{X_k,U}[m] = \mu_{U,X_k}[\check m]$ together with the identity $(\xi_k\partial_{\xi_l}a)(-\xi) = \xi_k\,\partial_{\xi_l}\check a(\xi)$ shows that replacing
$\mu_{U,X_k}$ by the symmetrised flux $\pi_k = \tfrac12(\mu_{U,X_k}+\mu_{X_k,U})$ leaves the sum unchanged, yielding the intrinsic form
\begin{equation*}
    \mathcal{R}^{(1)}+\mathcal{R}^{(2)}+\mathcal{R}^{(3)} = \tfrac12\sum_{k,l}\big\langle\pi_k,\ \rho\,\partial_{x_l}(c/\rho)\,\xi_k\partial_{\xi_l}a\,\Psi\big\rangle = \langle\mathcal{R}, a\Psi\rangle,
\end{equation*}
which is \eqref{eq:refraction}. When $\rho, c$ are constant each $\partial_{x_l}(c/\rho)$ vanishes, so $\mathcal{R}\equiv 0$ and \eqref{eq:wave-variable}
collapses to $\partial_t\nu = \operatorname{div}_x(c\boldsymbol{\pi})$ (with $c$ a positive constant scalar); in the normalisation $\rho = c = 1$ this is
exactly Theorem \ref{thm:bicharacteristic_pwave}. This establishes \eqref{eq:wave-variable}.
\end{proof}

\begin{remark}[Bicharacteristic reading and the classical $L^2$ status]\label{rem:wave-refraction}
Theorem \ref{thm:wave-variable} is the $x$-cosphere ($\tau$-marginal) projection of the classical transport of the wave H-measure along the null bicharacteristics
of the symbol $\rho\tau^2 = c|\xi|^2$ \cite{Tartar2017} (whose H-measure lives on the space--time cosphere, while $\nu, \boldsymbol{\pi}$ are its $x$-only
projections). The local squared wave speed is $c/\rho$, and the refraction functional $\mathcal{R}$ is driven precisely by its spatial gradient $\nabla(c/\rho)$,
the frequency-bending $\dot\xi_l \propto -\partial_{x_l}(c/\rho)$ of geometric optics. The transport is truly spatial: the microlocal energy is carried at
the group velocity through the flux $\operatorname{div}_x(c\boldsymbol{\pi})$ and is not stationary even in a homogeneous medium. As at $p = 2$ the block
$\boldsymbol{\mu}$ is the classical Tartar H-measure and $\Phi_2 = \operatorname{id}$, the result is not new as an $L^2$ statement; its role is to exhibit
the Second Commutation Lemma (Theorem \ref{thm:second_comm}) at work, commuting $\mathcal{A}_a$ past the variable multipliers $\rho, c$ and
reducing the resulting non-local commutators to the principal-symbol refraction plus compact remainders, in the one wave regime where the microlocal hierarchy
closes. The obstruction to the analogous $p\neq 2$ statement is isolated in Remark \ref{rem:nonlinear-obstruction}.
\end{remark}

\begin{remark}[The nonlinear obstruction for $p\neq 2$]\label{rem:nonlinear-obstruction}
For $p\neq 2$ the microlocal energy does not close within the block $\boldsymbol{\mu}$. The system prescribes the evolution of the dual momentum
$\Phi_p(U_n)$ (equation \eqref{eq:pwave2}) and of the strain $X_n$ (equation \eqref{eq:pwave1}), but the velocity $U_n$ itself carries no evolution
equation, only $\Phi_p(U_n)$ does. Consequently, differentiating the diagonal $\mu_U = $\,H-dist$(U_n, \Phi_p(U_n))$ in time yields, after integration
by parts and use of \eqref{eq:pwave1}--\eqref{eq:pwave2}, two irreducible contributions: a term pairing $\partial_t U_n$ (for which the system provides
no evolution equation) against $\mathcal{A}_a\Phi_p(U_n)$, and a time-differentiated cross term pairing $\partial_t X_{k,n} = \partial_{x_k}U_n$
against $\mathcal{A}_a\Phi_p(X_n)_k$. Neither is expressible through the components of $\boldsymbol{\mu}$: the true obstruction is the absence
from the block $\boldsymbol{\mu}$ of an evolution law for $U_n$ and of the time-differentiated cross pairs $(\partial_t X_n, \Phi_p(X_n))$ adjoined in Section
\ref{sec:future}, a sharper statement than the (correct but secondary) observation that $\mathcal{A}_a$ does not commute with $\Phi_p$. The hierarchy therefore
fails to close. The symmetrised Poynting-flux term $c\,\boldsymbol{\pi}$ is nonetheless correctly identified as the spatial energy flux, exactly as in the
linear case. Establishing a closed microlocal energy transport for the quasilinear $p$-wave system ($p\neq 2$), equivalently, identifying the enlarged
family of H-distributions (including those built from the time-derivatives $\partial_t W_n$) that closes the system, is an open problem, of a piece with
the anisotropic barrier of Remark \ref{rem:anisotropic}.
\end{remark}

\begin{remark}[The Anisotropic Barrier]\label{rem:anisotropic}
The restriction to an isotropic scalar $c(x)$ in the preceding system is a functional analytic necessity, not merely a convenience. For a general symmetric
anisotropic tensor $C(x) \in \mathcal{M}_{d\times d}$, the algebraic symmetry required for Lemma \ref{lemma:energy} breaks: the weighted spatial energy density
$|X|^{p-2}(CX)\cdot\partial_t X$ is an exact $t$-derivative only when $C$ is a scalar multiple of the identity, $C = cI$ (for then
$|X|^{p-2}(CX)\cdot\partial_t X = c\,|X|^{p-2}X\cdot\partial_t X = \tfrac{c}{p}\,\partial_t|X|^p$), whereas for genuinely anisotropic $C$ the Euclidean weight
$|X|^{p-2}$ and the quadratic form $CX\cdot X$ are mismatched and no such collapse occurs. This obstruction is proper to $p \neq 2$: at $p = 2$ the weight
is absent and every symmetric, time-independent $C$ yields the exact identity $(CX)\cdot\partial_t X = \tfrac12\partial_t(CX\cdot X)$, which is precisely
why the linear wave theory of Section \ref{sec:wave_eq} (Theorem \ref{thm:wave-variable}) accommodates variable, indeed in principle anisotropic, media.

To recover a valid local energy identity for generic anisotropic media, one must abandon the standard Euclidean Nemyckij operator and define an anisotropic
metric dual: $\Phi_{p,C}(X) = (C(x)X \cdot X)^{\frac{p-2}{2}} C(x)X$. However, doing so absorbs the spatial variable $C(x)$ into the nonlinear topology of
the dual sequence $L^{p'}$. Consequently, the pseudo-differential multiplier $\mathcal{A}_a$ no longer crosses $C(x)$ independently, preventing the Second
Commutation Lemma from extracting the refraction Poisson bracket $\{a, C_{jk}\}$. Thus, in global $L^p$ spaces, anisotropic geometric refraction is entangled
with the nonlinear dual topology, which remains an open problem for the generalisation of geometric optics.
\end{remark}

\section{Future Directions and Open Problems}\label{sec:future}

The global $L^p$ second commutation lemma, together with the algebraic stabilization provided by the canonical dual sequence, opens several directions in the
study of microlocal concentration. The functional analytic framework developed here provides tools for several open problems in partial differential equations;
two of these arise directly as structural limitations of the transport results proved above.

\subsection{Closing the Microlocal Energy Hierarchy for Quasilinear $p$-Waves}
Theorem \ref{thm:bicharacteristic_pwave} establishes the microlocal energy transport for the isotropic wave system only in the linear core $p = 2$, where the
primal and dual sequences coincide and the energy density lifts to a closed conservation law carried by the Poynting flux $\boldsymbol{\pi}$. For $p \neq 2$
this closure fails, and we identify a structural, not merely technical, obstruction to it (Remark \ref{rem:nonlinear-obstruction}): the system evolves the
dual momentum $\Phi_p(U_n)$ and the strain $X_n$, but not the velocity $U_n$ itself, while the Fourier multiplier $\mathcal{A}_a$ does not commute with the
nonlinear Nemyckij map $\Phi_p$. Differentiating the diagonal $\mu_U$ in time therefore generates a bilinear limit that is not expressible through the
components of the block $\boldsymbol{\mu}$.

The open problem is to identify the enlarged family of microlocal objects that closes this hierarchy. A natural candidate is to adjoin the cross H-distributions
coupling the time-differentiated sequences $\partial_t W_n$ to the dual momenta, the pairs $(\partial_t U_n, \Phi_p(U_n))$ and $(\partial_t X_n, \Phi_p(X_n))$.
This must be done with care: a uniform $L^p_{loc}$ bound on $\partial_t W_n$ would, through the compatibility relation $\partial_t X_n = \nabla U_n$,
place $U_n$ in a bounded set of $W^{1,p}_{loc}$ and hence force $U_n \to 0$ strongly (Rellich--Kondrachov, exactly as in Remark \ref{rem:novacuity}),
collapsing $\mu_U$ (and with it the entire Poynting vector $\boldsymbol{\pi}$) to zero. A viable enlargement must therefore renormalise, pairing
$\partial_t W_n / \omega_n$ against $\Phi_p(W_n)$ with $\omega_n$ a characteristic frequency, in the spirit of second-microlocal (semiclassical) objects;
the choice of normalisation is itself part of the open problem. One then seeks a finite system of transport equations relating these renormalised objects to
$\boldsymbol{\mu}$. Establishing such a closed hierarchy would upgrade the Poynting-flux law of Theorem \ref{thm:bicharacteristic_pwave} to a true
bicharacteristic transport with refraction for the full quasilinear $p$-wave system, and would decide whether the correct geometric optics of $p$-waves is
governed by a finite-rank system or by an infinite tower of microlocal moments.

\subsection{Decoupling in the $p$-Elasticity System (Cauchy-Lam\'e)}
Tartar \cite{Tartar2017} demonstrated that for the classical linearised elasticity system, the $L^2$ H-measure matrix decouples the propagation of energy
into longitudinal (P-waves) and transverse (S-waves) components. This decoupling is dictated by the eigenspaces of the acoustic tensor. However, as discussed
in Section \ref{sec:wave_eq}, linear hyperbolic systems lose their energetic isometry in the global $L^p$ setting. 

Future work must investigate the quasilinear $p$-Elasticity system. This requires pairing the velocity field $U^n = \partial_t u^n \in L_{loc}^p(\R^d;\R^m)$
and the strain $X^n = \nabla u^n$ (properly, its symmetric part $\tfrac12(\nabla u^n + (\nabla u^n)^\top)$) with \emph{frame-indifferent} canonical Nemyckij
duals \cite{AntonicErcegMisur2021} (the Euclidean $\Phi_p(U^n) = |U^n|^{p-2}U^n$ and the Frobenius $\Phi_p(X^n) = |X^n|^{p-2}X^n$) rather than
basis-dependent componentwise duals $|X_{ij}^n|^{p-2}X_{ij}^n$, which would entangle the P/S-decoupling question with the choice of axes. The central open
question is whether the resulting block H-distribution still permits a geometric decoupling of pressure and shear waves in the Banach space framework, or
if the non-Euclidean geometry ($p \neq 2$) algebraically entangles the characteristic varieties. For $p \neq 2$ this decoupling question is
downstream of, and conditional upon, the closure problem discussed above: the elasticity energy hierarchy inherits, in aggravated, coupled form, the
same non-closure obstruction as the scalar $p$-wave system (Remark \ref{rem:nonlinear-obstruction}), so a closed propagation must first be secured before
its decoupling into P- and S-waves can even be posed. 

Resolving this will require using the $L^p$ second commutation lemma to push pseudo-differential multipliers through the variable, fourth-order elasticity
tensor $C_{ij;kl}(x)$, and subsequently projecting the resulting Poisson brackets onto the generalized $p$-acoustic eigenspaces.

\subsection{Microlocal Analysis of Quasilinear $p$-Growth Problems}
Classical $L^2$ H-measures are inadequate for strongly nonlinear problems governed by $p$-growth energy functionals, which arise in non-Newtonian fluid
dynamics and large-strain nonlinear elasticity. The archetypal model for these systems is the $p$-Laplacian equation:
\begin{equation*}
    -\nabla \cdot (|\nabla u_n|^{p-2}\nabla u_n) = f_n.
\end{equation*}

The observation provided by the $L^p$-$L^{p'}$ duality framework is that the nonlinear bulk term inside the divergence, $|\nabla u_n|^{p-2}\nabla u_n$,
is the canonical dual sequence $\Phi_p(\nabla u_n)$. Consequently, the H-distribution $\mu$ generated by the sequence pair
$\big(\nabla u_n, \Phi_p(\nabla u_n)\big)$ represents the concentration of stress and strain energy within the nonlinear material. 

Because quasilinear PDEs lack the superposition principle, attempting to linearise these operators via standard microlocal techniques generates non-compact
remainders. The linearisation moreover produces coefficients $|\nabla u_n|^{p-2}\big(\mathrm{Id} + (p-2)\,\omega_n\otimes\omega_n\big)$
(with $\omega_n = \nabla u_n/|\nabla u_n|$) that are themselves $n$-dependent and only as regular as $\nabla u_n$, so a freezing (paralinearisation) step is
needed before the commutator calculus applies. However, the global $L^p$ second commutation lemma provides the fractional pseudo-differential bounds required
to bypass this barrier. 

Using the fractional smoothing into $W^{\epsilon',p}(\R^d)$ ($\epsilon' < \epsilon$) established in Lemma \ref{lem:remainder}, future research can control
the commutators generated by the nonlinear terms. This provides harmonic analysis tools for establishing sharper microlocal localisation and
compactness-by-compensation principles for degenerate quasilinear equations.

\subsection{The Low-Regularity Range $1 < p < 2$}
The transport results of Section \ref{sec:applications} are proved for $p \ge 2$ (Theorem \ref{thm:bicharacteristic_first_order}), and this restriction is
not merely technical. The derivation relies on the canonical Nemyckij dual $v_n = \Phi_p(u_n)$ lying in $W^{1,p'}_{loc}$ with the pointwise chain rule
$\nabla v_n = (p-1)|u_n|^{p-2}\nabla u_n$; for $1 < p < 2$ the derivative $\Phi_p'$ is singular at the origin, and near a transversal zero of $u_n$ this
identity yields an $L^{p'}_{loc}$ function only when $p > \tfrac{1+\sqrt{5}}{2}$ (Remark \ref{rem:chainrange}). The genuinely Banach regime, in which the
dual sequence is merely H\"older continuous, is thus left open. A related but milder difficulty affects the $p$-Wave system of Section \ref{sec:wave_eq}:
the local energy identity of Lemma \ref{lemma:energy} was derived by differentiating $\Phi_p(U_n)$ and $\Phi_p(X_n)$ directly, which is unconditional only for
$p \ge 2$. There, however, the identity itself involves only the $C^1$ densities $|U_n|^p, |X_n|^p$ (all terms $L^1_{loc}$ for $p > 1$ at transversal zeros),
so, unlike the scalar dual, which requires $\nabla\Phi_p(u_n) \in L^{p'}_{loc}$ and hence $p > \tfrac{1+\sqrt5}{2}$, it is expected to survive
throughout $1 < p < \infty$ once the differentiation is justified by a nodal-set (defect) analysis. A single resolution of the low-regularity Nemyckij calculus
would place both applications on the same rigorous footing.

Resolving it will require a substitute for the chain rule adapted to the low-regularity dual. Promising routes include a defect analysis of the smooth
truncations $\Phi_p^\delta(s) = (s^2+\delta^2)^{(p-2)/2}s$ as $\delta \to 0$, controlling the contribution concentrated on the nodal set $\{u_n = 0\}$;
a restriction to sequences whose gradients do not concentrate on their nodal sets; or a renormalised dual that regularises the singularity of $\Phi_p$
while preserving the $L^p$--$L^{p'}$ duality pairing. Any of these would extend both the bicharacteristic flow of Theorem \ref{thm:bicharacteristic_first_order}
and the energy identity of Lemma \ref{lemma:energy} to the full range $1 < p < \infty$.

\subsection{Sharp Symbol Regularity via Littlewood--Paley}
The kernel-based proof of Lemma \ref{lem:remainder} in Appendix \ref{app:remainder} is deliberately elementary, and its price is the symbol-regularity threshold
$s > d+2$ of Assumption \ref{ass:symbol}, dictated by the pointwise Calder\'on--Zygmund kernel bounds (Remark \ref{rem:sobolev}). As anticipated in Remark
\ref{rem:sharper}, a frequency-space argument (decomposing the commutator by the Bony paraproduct calculus and estimating each dyadic block of
$\mathcal{A}_a$ through the H\"ormander--Mikhlin theorem) should establish the same order-$(-\epsilon)$ smoothing under the substantially weaker
hypothesis $s > \tfrac{d+4}{2}$, matching the H\"ormander boundedness threshold plus the two orders consumed by the second-order Taylor remainder.
Carrying out this paraproduct proof, and thereby identifying the sharp interplay between symbol regularity, the smoothing order $-\epsilon$, and the dimension
$d$, is a concrete open problem; it would in particular decide whether $s > \tfrac{d+4}{2}$ is optimal or can be lowered further towards the pure boundedness
threshold $\tfrac{d}{2}$.

\subsection{Propagation of Singularities in Fractional and Anomalous PDEs}

While classical microlocal analysis was developed for partial differential operators of integer order, modern physical models for anomalous diffusion,
non-Markovian stochastic processes, and non-local continuum mechanics increasingly rely on fractional differential operators, such as the fractional
Laplacian $(-\Delta)^s$. 

Extending the theory of H-distributions to fractional PDEs presents a significant challenge. Fractional operators are non-local; they do not satisfy the
standard local Leibniz rule, and commuting them with variable coefficients generates integral tails that extend to spatial infinity. Consequently, classical
commutation lemmas that rely on local geometric arguments fail to capture the asymptotic behavior of weakly converging sequences in these anomalous models. 

The functional analytic framework developed in this manuscript is suited to this problem. The second commutation lemma's remainder is an order-$(-\epsilon)$
smoothing operator (Lemma \ref{lem:remainder}), so its estimates map $L^p$ into the fractional Sobolev spaces $W^{\epsilon',p}(\R^d)$ ($\epsilon' < \epsilon$);
and the spatial truncation technique developed in Section \ref{sec:theorem_proof} is suited to the non-local tails generated by operators on the unbounded
domain. It is here, rather than in the transport applications of Sections \ref{sec:applications}--\ref{sec:wave_eq}, which localise exactly
(\S\ref{sec:comparison}), that the global reach of the lemma becomes indispensable: localising $[(-\Delta)^s, M_\varphi]$ leaves a kernel tail
spread over all of $\R^d$, so the pairing probes the coefficients at infinity and no merely local commutation lemma would suffice. We caution that
the second commutation lemma governs the degree-zero probe multipliers $\mathcal{A}_a$ rather than the fractional operator $(-\Delta)^s$ itself (a Fourier
multiplier of degree $2s$); the programme below commutes these probes through the variable coefficients of the fractional PDE.

Future research will use these fractional commutator bounds to advance beyond the existence and localisation theories established for H-distributions by
Antoni\'c, Erceg, and Mi\v{s}ur \cite{AntonicErcegMisur2021}. Specifically, by commuting pseudo-differential multipliers through the variable coefficients
of fractional PDEs, it will be possible to derive the transport equations that govern how H-distributions propagate through anomalous, non-local media,
establishing a geometric optics theory for non-local operators in Banach spaces.

\subsection{Metric-Adapted Scalings for Anisotropic Media}

As observed in Remark \ref{rem:anisotropic}, the derivation of transport equations for the quasilinear $p$-Wave system currently requires an isotropic scalar
restriction. This is because the standard canonical Nemyckij operator $\Phi_p(X) = |X|^{p-2}X$ is tied to the Euclidean metric. Attempting to define an
anisotropic dual operator $\Phi_{p,C}(X)$ entangles the spatial tensor $C(x)$ within the nonlinear topology, shielding it from the pseudo-differential
commutator and halting the extraction of the geometric refraction rays.

One avenue to resolve this anisotropic barrier is to alter the underlying projective geometry of the H-distribution itself. This strategy has been used
in other domains where standard microlocal scaling fails; for instance, Antoni\'c and Lazar established parabolic H-measures by replacing the standard
projection onto the unit sphere $\Sp^{d-1}$ with a projection onto a parabolic manifold $P^d$ to align with the anisotropic time-space scaling of the heat
operator \cite{AntonicLazar2013}.

To resolve anisotropic quasilinear systems, future research must investigate metric-adapted H-distributions. Rather than scaling along standard Euclidean rays,
the frequency space could be projected onto a variable, metric-dependent manifold (e.g., the characteristic variety defined by $C(x)\xi \cdot \xi = 1$).
By warping the projective scaling of the H-distribution to match the geometry of the acoustic tensor, the dual sequence might remain decoupled from the
spatial variations. Because this base manifold would vary with $x$ (a departure from the fixed-manifold constructions of Tartar and Antoni\'c--Lazar,
and one whose very feasibility is itself an open question), the programme is speculative; should the global $L^p$ second commutation lemma prove adaptable
to such metric-dependent projective manifolds, it would yield the full phase-space transport equations for anisotropic, non-Euclidean nonlinear media.

\subsection{Structure and Intrinsicness of $L^p$ H-Distributions}
Two questions underpin the transport theory developed here. The first concerns the admissible regularity of the coefficients. The zeroth-order coefficient
in Theorem \ref{thm:bicharacteristic_first_order} was required to be bounded and continuous ($c \in C_b$), so that it may be pulled into the H-distribution
pairing against the weak-$*$ limit of the associated $L^1$ densities. Whether the transport equations persist for rough coefficients ($c$ of
vanishing mean oscillation, or $c \in L^\infty$ endowed with additional structure) is open, and would connect the present framework to compensated
compactness with discontinuous coefficients.

The second question concerns the intrinsic nature of the microlocal object. Unlike an $L^2$ H-measure, which is a positive Radon measure canonically attached
to the sequence, the $L^p$ H-distribution ($p \neq 2$) lacks positivity and is a distribution of finite anisotropic order; moreover it was generated here
through a specific device, the canonical Nemyckij dual $\Phi_p(u_n)$. It is not known whether the resulting propagation law is \emph{intrinsic} (that is,
independent of the admissible dual sequence used to define $\mu$) or whether distinct pairings yield different transport. Establishing such
choice-independence would endow the $L^p$ theory with the coordinate-free invariance that Hilbert-space structure confers automatically in $L^2$, and is
arguably the central structural problem left open by this work.

\section*{Acknowledgements}
The author gratefully acknowledges Ljudevit Palle for their insightful comments and engaging discussions, which significantly contributed to the final version
of the manuscript.

\section*{Declarations}
\textbf{Funding:} The author declares that no funds, grants, or other support were received during the preparation of this manuscript. \\
\textbf{Conflict of interest:} The author declares that he has no conflict of interest. \\
\textbf{Data availability:} Data sharing is not applicable to this article as no datasets were generated or analysed during the current study.\\
\textbf{Generative AI and AI-assisted technologies:} During the preparation of this work, the author(s) used Google Gemini as an assistive tool to transcribe handwritten mathematical notes into \LaTeX{} formatting, to polish the English prose for readability, and to help draft the initial abstract and manuscript summary. Additionally, the AI was utilized during the research phase for conceptual exploration, specifically to search heuristically for potential counterexamples to stress-test preliminary hypotheses. After using this tool, the author(s) meticulously reviewed, verified, and edited all generated text and code. All mathematical claims, proofs, and counterexamples were independently rigorously verified by the human author(s). The author(s) take full intellectual responsibility for the final content of this publication, including all mathematical proofs, formatting, and conceptual framing.

\appendix

\section{Proof of the Non-smooth Symbolic Remainder Estimate}\label{app:remainder}

We prove Lemma \ref{lem:remainder} by elementary means (classical Calder\'on--Zygmund kernel bounds, Taylor's theorem, and the finite-difference
characterisation of Besov spaces), the sole external input being the $L^p$-boundedness of the commutator from Part I of Theorem \ref{thm:second_comm}
(itself resting on Calder\'on's first commutator theorem), invoked once in \S\ref{app:lp} to control the $L^p$ component of the Besov norm. Constants
denoted $C$ depend only on $d$, $p$, $\epsilon$ and $\norm{a}_{W^{s,2}(\Sp^{d-1})}$, and may change from line to line; we abbreviate
$[\nabla b]_\epsilon := [\nabla b]_{C^\epsilon(\R^d)}$ and write $\Delta_h f(x) := f(x+h) - f(x)$.

\subsection{Kernel bounds for the multiplier}\label{app:kernel}

By Assumption \ref{ass:symbol}, $a$ is homogeneous of degree zero with $a|_{\Sp^{d-1}} \in W^{s,2}(\Sp^{d-1})$, $s > d+2$; the symbol $i\xi_j a(\xi)$ of
$\partial_j\mathcal{A}_a$ is then homogeneous of degree $1$, and its convolution kernel $\partial_j k$ (with $k = \mathcal{F}^{-1}a$) is homogeneous of
degree $-d-1$ off the origin. To convert the symbol's Sobolev regularity into pointwise kernel bounds we use the Bochner--Stein--Weiss identity on
spherical harmonics \cite[Ch.~IV]{SteinWeiss1971}: for a solid harmonic $P_\ell$ of degree $\ell$,
\begin{equation*}
    \mathcal{F}\big[P_\ell(x)\,|x|^{-d-\ell+\alpha}\big](\xi) = \gamma_{\ell,\alpha}\,P_\ell(\xi)\,|\xi|^{-\ell-\alpha}, \qquad \gamma_{\ell,\alpha} \sim c\,\ell^{\,\alpha - d/2} \quad (\ell \to \infty).
\end{equation*}
We convert this into a Sobolev bookkeeping on the sphere. Since $m(\xi) := i\xi_j a(\xi)$ is homogeneous of degree $1$, its restriction to the sphere factorises
as $m(\xi) = |\xi|\,\tilde m(\xi/|\xi|)$ with $\tilde m := m|_{\Sp^{d-1}} = \sum_{\ell \ge 0} M_\ell$, where $M_\ell$ is the degree-$\ell$ spherical-harmonic
component; multiplication by the coordinate $\xi_j$ (itself a degree-one spherical harmonic) is bounded on every $W^{\tau,2}(\Sp^{d-1})$, so
$\norm{\tilde m}_{W^{\tau,2}(\Sp^{d-1})} \le C\norm{a}_{W^{\tau,2}(\Sp^{d-1})}$ for all $\tau \le s$. Applying the identity above term by term with
$\alpha = -1$ (so that the degree-$(-d+\alpha) = -(d+1)$ homogeneous piece $M_\ell(z/|z|)\,|z|^{-d-1}$ has Fourier transform
$\gamma_{\ell,-1}\,M_\ell(\xi/|\xi|)\,|\xi|$, and inverting, the degree-$1$ symbol $M_\ell(\xi/|\xi|)|\xi|$ has kernel
$\gamma_{\ell,-1}^{-1}M_\ell(z/|z|)|z|^{-d-1}$), the convolution kernel acquires the angular expansion
\begin{equation*}
    \partial_j k(z) = |z|^{-d-1}\,\Theta(z/|z|), \qquad \Theta = \sum_{\ell \neq 1} \gamma_{\ell,-1}^{-1}\, M_\ell,
\end{equation*}
the term $\ell = 1$ being omitted: there $\gamma_{\ell,-1}$ degenerates, but the corresponding component of $\tilde m$ is linear in $\xi$, a polynomial whose
inverse Fourier transform is a combination of first-order derivatives of $\delta$ supported at the origin and hence invisible to the off-origin kernel.
Because $\gamma_{\ell,-1}^{-1} \sim c\,\ell^{\,d/2+1}$, the Laplace--Beltrami characterisation
$\norm{f}_{W^{\sigma,2}(\Sp^{d-1})}^2 \simeq \sum_\ell (1+\ell)^{2\sigma}\norm{Y_\ell}_{L^2}^2$ (with $Y_\ell$ the harmonic components of $f$) gives
\begin{equation*}
    \norm{\Theta}_{W^{\sigma,2}(\Sp^{d-1})} \le C\,\norm{\tilde m}_{W^{\sigma + \frac{d}{2}+1,\,2}(\Sp^{d-1})} \le C\,\norm{a}_{W^{\sigma + \frac{d}{2}+1,\,2}(\Sp^{d-1})},
\end{equation*}
so passing from the symbol's spherical data to the kernel's costs exactly $\tfrac{d}{2} + 1$ derivatives in the $W^{\sigma,2}(\Sp^{d-1})$ scale.
Landing the angular profile $\Theta$ in $C^0(\Sp^{d-1})$ through the embedding $W^{\sigma,2}(\Sp^{d-1}) \hookrightarrow C^0(\Sp^{d-1})$
($\sigma > \tfrac{d-1}{2}$) therefore requires $s - (\tfrac{d}{2} + 1) > \tfrac{d-1}{2}$, i.e.\ $s > d + \tfrac{1}{2}$, and one further tangential derivative
for $\Theta \in C^1(\Sp^{d-1})$ (needed for the gradient bound), i.e.\ $s > d + \tfrac{3}{2}$; both are secured by Assumption \ref{ass:symbol}.
Since $\partial_j k$ is homogeneous of degree $-d-1$ and $\nabla\partial_j k$ of degree $-d-2$, this yields, for $z \neq 0$,
\begin{equation}\label{eq:app-kernel}
    |\partial_j k(z)| \le C\,|z|^{-d-1}, \qquad |\nabla \partial_j k(z)| \le C\,|z|^{-d-2}.
\end{equation}
Moreover, for each $l$ the function $z \mapsto z_l\,\partial_j k(z)$ is homogeneous of degree $-d$ and, by the computation \eqref{eq:app-moment} below,
is the Calder\'on--Zygmund kernel of a bounded operator on $L^p(\R^d)$.

\subsection{The exact kernel identity}\label{app:identity}

We first normalise the symbol. Writing $c_0 := |\Sp^{d-1}|^{-1}\int_{\Sp^{d-1}}a\,\rd\sigma$ for its spherical mean,
$k = \mathcal{F}^{-1}a = c_0\,\delta + \mathrm{p.v.}\,k_0$, where $k_0 = \mathcal{F}^{-1}(a - c_0)$ is a Calder\'on--Zygmund kernel homogeneous of
degree $-d$ with vanishing spherical mean. Since $[\mathcal{A}_a, M_b]$, $S_j$, and hence $R_j(b)$ are all invariant under $a \mapsto a - c_0$
(constants commute with $M_b$, and $\partial_{\xi_k}(a-c_0) = \partial_{\xi_k}a$), we assume henceforth $c_0 = 0$, so that $k = \mathrm{p.v.}\,k_0$;
every display below is then a genuine principal-value integral of the function $k_0$, and \eqref{eq:app-moment} is an identity of tempered
distributions: in particular the moment operator $P_l$ has exactly the symbol $\delta_{jl}a + \xi_j\partial_{\xi_l}a$, any local (delta)
contribution of the distributional product $z_l\,\partial_j k$ being already encoded in that symbol.

Let $u \in C_c^\infty(\R^d)$. Since $\mathcal{A}_a u = k * u$,
\begin{equation*}
    [\mathcal{A}_a, M_b]u(x) = \int_{\R^d} k(x-y)\big(b(y) - b(x)\big) u(y)\,\rd y,
\end{equation*}
the integral being absolutely convergent near the diagonal because $|k(x-y)| \le C|x-y|^{-d}$ and $|b(y)-b(x)| \le \norm{\nabla b}_\infty |x-y|$.
Differentiating in $x_j$,
\begin{equation}\label{eq:app-diff}
    \partial_j[\mathcal{A}_a, M_b]u(x) = \mathrm{p.v.}\!\int (\partial_j k)(x-y)\big(b(y)-b(x)\big)u(y)\,\rd y \;-\; \partial_j b(x)\,\mathcal{A}_a u(x).
\end{equation}
Write $b(y) - b(x) = \nabla b(x)\cdot(y-x) + r(x,y)$, where the Taylor remainder satisfies
\begin{equation}\label{eq:app-taylor}
    r(x,y) = \int_0^1 \big[\nabla b(x + t(y-x)) - \nabla b(x)\big]\cdot(y-x)\,\rd t, \qquad |r(x,y)| \le \tfrac{[\nabla b]_\epsilon}{1+\epsilon}\,|x-y|^{1+\epsilon}.
\end{equation}
For the linear part, we compute the Fourier multiplier of the convolution kernel $z_l\,\partial_j k(z)$:
\begin{equation}\label{eq:app-moment}
    \mathcal{F}\big[z_l\,\partial_j k\big](\xi) = i\,\partial_{\xi_l}\mathcal{F}[\partial_j k](\xi) = i\,\partial_{\xi_l}\big(i\xi_j a(\xi)\big) = -\big(\delta_{jl}\,a(\xi) + \xi_j\,\partial_{\xi_l}a(\xi)\big).
\end{equation}
Setting $P_l u(x) := \mathrm{p.v.}\!\int (\partial_j k)(x-y)(y_l - x_l)u(y)\,\rd y = -\big((z_l\,\partial_j k)*u\big)(x)$,
identity \eqref{eq:app-moment} gives $P_l = \operatorname{Op}\big(\delta_{jl} a + \xi_j\partial_{\xi_l}a\big)$, whence
\begin{equation}\label{eq:app-linear}
    \sum_{l=1}^d \partial_{x_l}b(x)\,P_l u(x) = \partial_j b(x)\,\mathcal{A}_a u(x) + \sum_{l=1}^d \partial_{x_l}b(x)\,\operatorname{Op}(\xi_j\partial_{\xi_l}a)u(x) = \partial_j b(x)\,\mathcal{A}_a u(x) + S_j u(x).
\end{equation}
Substituting the split of $b(y)-b(x)$ into \eqref{eq:app-diff} and using \eqref{eq:app-linear},
\begin{equation*}
    \partial_j[\mathcal{A}_a, M_b]u = \big(\partial_j b\,\mathcal{A}_a u + S_j u + \mathcal{R}u\big) - \partial_j b\,\mathcal{A}_a u = S_j u + \mathcal{R}u,
\end{equation*}
where $\mathcal{R}u(x) := \int (\partial_j k)(x-y)\,r(x,y)\,u(y)\,\rd y$ is absolutely convergent by \eqref{eq:app-kernel}--\eqref{eq:app-taylor}.
Therefore $R_j(b) = \mathcal{R}$ on $C_c^\infty(\R^d)$, which is the kernel representation asserted in the lemma; the operator
$\mathcal{R}$ (defined by this absolutely convergent integral on $C_c^\infty$) extends to a bounded operator on $L^p(\R^d)$ by the
estimate of \S\ref{app:lp} (the pointwise kernel integral need not converge absolutely for general $u \in L^p$: the size bound
$|x-y|^{-d+\epsilon}$ is locally integrable but is \emph{never} globally in $L^{p'}_y(\R^d)$, its far-field tail lying in $L^{p'}$
only for $p < d/\epsilon$ and its local singularity only for $p > d/\epsilon$, so no single global H\"older estimate applies and
absolute convergence is unavailable), and $R_j(b) = \mathcal{R}$ holds throughout by density.
Writing $\mathcal{K}(x,y) := (\partial_j k)(x-y)\,r(x,y)$ for its kernel, \eqref{eq:app-kernel} and \eqref{eq:app-taylor} give the global size bound
\begin{equation}\label{eq:app-size}
    |\mathcal{K}(x,y)| \le C\,[\nabla b]_\epsilon\,|x-y|^{-d+\epsilon}, \qquad x \neq y.
\end{equation}

\subsection{$L^p$ boundedness}\label{app:lp}

By Part I of Theorem \ref{thm:second_comm}, $\partial_j[\mathcal{A}_a, M_b]$ is bounded on $L^p(\R^d)$ with norm $\le C\norm{\nabla b}_\infty$.
The operator $S_j = \sum_l (\partial_{x_l}b)\,\operatorname{Op}(\xi_j\partial_{\xi_l}a)$ is a finite sum of Calder\'on--Zygmund
operators (bounded on $L^p$ by the H\"ormander multiplier theorem, since each $\xi_j\partial_{\xi_l}a$ is homogeneous of degree zero and
inherits the required spherical regularity from Assumption \ref{ass:symbol}), followed by multiplication by $\partial_{x_l}b \in L^\infty$. Hence
\begin{equation}\label{eq:app-lpbound}
    \norm{R_j(b)}_{\mathcal{L}(L^p)} \le \norm{\partial_j[\mathcal{A}_a, M_b]}_{\mathcal{L}(L^p)} + \norm{S_j}_{\mathcal{L}(L^p)} \le C\norm{\nabla b}_\infty \le C\norm{\nabla b}_{C^\epsilon}.
\end{equation}

\subsection{The fractional smoothing estimate}\label{app:smoothing}

For $0 < \epsilon < 1$ the Besov space $B^\epsilon_{p,\infty}(\R^d)$ is characterised, up to norm equivalence, by \cite[\S 2.5.12]{Triebel1983}
\begin{equation*}
    \norm{f}_{B^\epsilon_{p,\infty}} \simeq \norm{f}_{L^p} + \sup_{0 < |h| \le 1} |h|^{-\epsilon}\norm{\Delta_h f}_{L^p}.
\end{equation*}
In view of \eqref{eq:app-lpbound}, the boundedness $R_j(b): L^p \to B^\epsilon_{p,\infty}$ follows once we prove
\begin{equation}\label{eq:app-goal}
    \norm{\Delta_h\, R_j(b)u}_{L^p} \le C\,[\nabla b]_\epsilon\,|h|^\epsilon\,\norm{u}_{L^p}, \qquad 0 < |h| \le 1.
\end{equation}
Since $\Delta_h R_j u(x) = \int G_h(x,y)u(y)\,\rd y$ with $G_h(x,y) = \mathcal{K}(x+h,y) - \mathcal{K}(x,y)$, we split the $y$-integration at the scale $|h|$.

\medskip
\emph{Near region $\{|x-y| \le 2|h|\}$.} Here we use the size bound \eqref{eq:app-size} for each term of $G_h$.
Since $\{|x-y|\le 2|h|\} \subset \{|x+h-y| \le 3|h|\}$,
\begin{equation*}
    \sup_x \int_{|x-y|\le 2|h|} |G_h(x,y)|\,\rd y \le C[\nabla b]_\epsilon \int_{|z|\le 3|h|} |z|^{-d+\epsilon}\,\rd z = C[\nabla b]_\epsilon\,|h|^\epsilon,
\end{equation*}
and by symmetry the same bound holds for $\sup_y \int_{|x-y|\le 2|h|}|G_h(x,y)|\,\rd x$.
Schur's test yields $\norm{N_h}_{\mathcal{L}(L^p)} \le C[\nabla b]_\epsilon|h|^\epsilon$, where $N_h$ denotes the near part.

\medskip
\emph{Far region $\{|x-y| > 2|h|\}$.} Write the far part as $F_h = F_h^{\mathrm{I}} + F_h^{\mathrm{II}}$ according to
\begin{equation*}
    G_h(x,y) = \underbrace{\big[(\partial_j k)(x+h-y) - (\partial_j k)(x-y)\big]\,r(x+h,y)}_{\text{(I)}} + \underbrace{(\partial_j k)(x-y)\,\big[r(x+h,y) - r(x,y)\big]}_{\text{(II)}}.
\end{equation*}
For (I): when $|x-y| > 2|h|$ the segment $[x-y,\,x+h-y]$ avoids the origin, with $|x-y+th| \ge \tfrac12|x-y|$, so the mean value theorem and
\eqref{eq:app-kernel} give $|(\partial_j k)(x+h-y)-(\partial_j k)(x-y)| \le C|h|\,|x-y|^{-d-2}$; together with
$|r(x+h,y)| \le C[\nabla b]_\epsilon|x-y|^{1+\epsilon}$ (from \eqref{eq:app-taylor} and $|x+h-y|\le\tfrac32|x-y|$) this yields
$|\text{(I)}| \le C[\nabla b]_\epsilon|h|\,|x-y|^{-d-1+\epsilon}$. Hence
\begin{equation*}
    \sup_x \int_{|x-y|>2|h|} |\text{(I)}|\,\rd y \le C[\nabla b]_\epsilon\,|h|\int_{|z|>2|h|}|z|^{-d-1+\epsilon}\,\rd z = C[\nabla b]_\epsilon\,|h|\cdot|h|^{-1+\epsilon} = C[\nabla b]_\epsilon\,|h|^\epsilon,
\end{equation*}
and symmetrically in $x$; Schur's test gives $\norm{F_h^{\mathrm{I}}}_{\mathcal{L}(L^p)} \le C[\nabla b]_\epsilon|h|^\epsilon$.

For (II) we use the algebraic identity
\begin{equation}\label{eq:app-rincrement}
    r(x+h,y) - r(x,y) = c_h(x) - \Delta_h\nabla b(x)\cdot(y-x), \qquad c_h(x) := \nabla b(x+h)\cdot h - \big[b(x+h)-b(x)\big],
\end{equation}
where $|c_h(x)| \le \tfrac{1}{1+\epsilon}[\nabla b]_\epsilon|h|^{1+\epsilon}$ (Taylor's theorem applied between $x$ and $x+h$) and
$|\Delta_h\nabla b(x)| \le [\nabla b]_\epsilon|h|^\epsilon$. Accordingly, on the far region
$F_h^{\mathrm{II}} = c_h\cdot\mathcal{T}_h u + \sum_l (\Delta_h\partial_l b)\, P_l^{(h)}u$, where
\begin{equation*}
    \mathcal{T}_h u(x) := \int_{|x-y|>2|h|}(\partial_j k)(x-y)u(y)\,\rd y, \qquad P_l^{(h)}u(x) := \int_{|x-y|>2|h|}(\partial_j k)(x-y)(x_l - y_l)u(y)\,\rd y.
\end{equation*}
The operators $P_l^{(h)}$ are truncations of the Calder\'on--Zygmund operator $-P_l$ of \eqref{eq:app-moment} (the sign immaterial to what follows),
whose kernel $z_l\,\partial_j k_0$ is homogeneous of degree $-d$ and satisfies the full Calder\'on--Zygmund package: the size bound
$|z_l\partial_j k_0(z)| \le C|z|^{-d}$ and gradient bound $|\nabla(z_l\partial_j k_0)(z)| \le C|z|^{-d-1}$ (from \eqref{eq:app-kernel}),
a vanishing spherical mean (its symbol $\delta_{jl}a + \xi_j\partial_{\xi_l}a$ is bounded, precluding a logarithm), and $L^2$-boundedness
(bounded symbol). By Cotlar's inequality each truncation is dominated pointwise by the maximal singular integral, bounded on $L^p$,
so $\sup_{h}\norm{P_l^{(h)}}_{\mathcal{L}(L^p)} \le C$ (Grafakos \cite[Theorem 5.3.5 and Corollary 5.3.7]{Grafakos2014}; see also Stein
\cite[Chapter VI]{Stein1993}). Therefore
\begin{equation*}
    \Big\| \sum_l (\Delta_h\partial_l b)\, P_l^{(h)}u \Big\|_{L^p} \le \norm{\Delta_h\nabla b}_{L^\infty}\sum_l \norm{P_l^{(h)}u}_{L^p} \le C[\nabla b]_\epsilon|h|^\epsilon\norm{u}_{L^p}.
\end{equation*}
For the remaining term, $\mathcal{T}_h$ is convolution with $(\partial_j k)\mathbf{1}_{\{|z|>2|h|\}} \in L^1(\R^d)$, whose $L^1$ norm is
$\le C\int_{|z|>2|h|}|z|^{-d-1}\rd z = C|h|^{-1}$ by \eqref{eq:app-kernel}; Young's inequality gives $\norm{\mathcal{T}_h}_{\mathcal{L}(L^p)} \le C|h|^{-1}$.
Combined with the bound on $c_h$,
\begin{equation*}
    \norm{c_h\cdot\mathcal{T}_h u}_{L^p} \le \norm{c_h}_{L^\infty}\norm{\mathcal{T}_h u}_{L^p} \le C[\nabla b]_\epsilon|h|^{1+\epsilon}\cdot|h|^{-1}\norm{u}_{L^p} = C[\nabla b]_\epsilon|h|^\epsilon\norm{u}_{L^p}.
\end{equation*}
Here the order-one blow-up $|h|^{-1}$ of the truncated singular integral $\mathcal{T}_h$ is compensated exactly by the second-order smallness $|h|^{1+\epsilon}$
of $c_h$; this is the sole point where the two derivatives of regularity carried by $b \in C^{1,\epsilon}$ are consumed.

Collecting the near part $N_h$ together with $F_h^{\mathrm{I}}$ and $F_h^{\mathrm{II}}$ establishes \eqref{eq:app-goal}. With \eqref{eq:app-lpbound} this
proves $R_j(b): L^p(\R^d) \to B^\epsilon_{p,\infty}(\R^d)$ with the stated norm bound. The continuous embedding
$B^\epsilon_{p,\infty}(\R^d) \hookrightarrow B^{\epsilon'}_{p,p}(\R^d) = W^{\epsilon',p}(\R^d)$ for every $\epsilon' \in (0,\epsilon)$
\cite[\S 2.3.2]{Triebel1983} then completes the proof of Lemma \ref{lem:remainder}.

\begin{remark}[On the target regularity]\label{rem:target}
The finite-difference argument delivers the endpoint Besov space $B^\epsilon_{p,\infty}$, hence the Sobolev--Slobodeckij space $W^{\epsilon',p}$
for every $\epsilon' < \epsilon$. This suffices in full for the compactness argument of Section \ref{sec:theorem_proof}, which only requires the
strictly positive fractional gain $\epsilon' > 0$ to activate the Rellich--Kondrachov embedding. Landing in the exact space
$W^{\epsilon,p} = B^\epsilon_{p,p}$ would require the $p$-summable refinement of \eqref{eq:app-goal} over dyadic $|h|$, which we do not pursue.
\end{remark}

\end{document}